\documentclass{amsart}
\usepackage[utf8]{inputenc}
\usepackage{amsmath,amssymb,graphicx,setspace,verbatim,url}
\usepackage{color}
\usepackage{hyperref}
\usepackage{multicol}
\usepackage{enumitem}
\usepackage{float}
\usepackage{soul}
\usepackage{tikz}
\usepackage[normalem]{ulem}
\usepackage[margin=2.5cm]{geometry}

\usepackage{tikz-cd}
\usetikzlibrary{arrows.meta}

\tikzcdset{
  arrow style={tikz,diagrams={>=Triangle}}
}

\input xy
\xyoption{all}
\newcommand*{\nfrac}[2]{\genfrac{}{}{0pt}{}{#1}{#2}}

\newtheorem{thmalpha}{Theorem}

\newtheorem{prop}{Proposition}[section]
\newtheorem{coro}[prop]{Corollary}
\newtheorem{thm}[prop]{Theorem}
\newtheorem{lemma}[prop]{Lemma}

\newtheorem{definition}[prop]{Definition}
\newtheorem{construction}[prop]{Construction}
\newtheorem{example}[prop]{Example}

\DeclareMathOperator{\Cor}{Cor}

\newcommand{\GG}{\mathcal{G}}
\newcommand{\CC}{\mathcal{C}}
\newcommand{\FF}{\mathcal{F}}
\newcommand{\PP}{\mathcal{P}}
\newcommand{\Aut}{\mathrm{Aut}}

\newcommand{\BP}{\mathcal{BP}}

\begin{document}
\title{Constructing new geometries: a generalized approach to halving for hypertopes}

\author[Claudio Alexandre Piedade]{Claudio Alexandre Piedade}
\address{Claudio Alexandre Piedade, Centro de Matemática da Universidade do Porto, Universidade do Porto, Portugal, Orcid number 0000-0002-0746-5893
}
\email{claudio.piedade@fc.up.pt}

\author[Philippe Tranchida]{Philippe Tranchida}
\address{Philippe Tranchida, Max Planck Institute for Mathematics in the Sciences, Inselstrasse 22 D-04107 Leipzig, Orcid number 0000-0003-0744-4934.
}
\email{tranchida.philippe@gmail.com}
{}

\thanks{\textit{Declaration of interests:} none}
\date{\today}
\maketitle

\begin{abstract}
Given a residually connected incidence geometry $\Gamma$ that satisfies two conditions, denoted $(B_1)$ and $(B_2)$, we construct a new geometry $H(\Gamma)$ with properties similar to those of $\Gamma$. This new geometry $H(\Gamma)$ is inspired by a construction of Percsy, Percsy and Leemans \cite{Leemans2000}. We show how $H(\Gamma)$ relates to the classical halving operation on polytopes, allowing us to generalize the halving operation to a broader class of geometries, that we call non-degenerate leaf hypertopes. Finally, we apply this generalization to cubic toroids in order to generate new examples of regular hypertopes.
\end{abstract}













\section{Introduction}

The study of polytopes is one of the foundations of mathematics, tracing back to the ancient Greeks.
The idea of an abstract polytope was formally defined in \cite{ARP} as a ranked poset whose elements are faces, generalizing the concept of the face-lattice of convex polytopes. This idea was developed on many fronts, with inspiration from the works of Tits~\cite{Tits1961} (incidence geometries on Coxeter groups), Buekenhout~\cite{buekenhout2013diagram} (diagram geometries), McMullen~\cite{McMullen_1967} (combinatorially regular convex polytopes), Grünbaum~\cite{grunbaum1976} (regular polystromata), Danzer and Schulte~\cite{Danzer1982} (regular incidence complexes). The automorphism group of abstract regular polytopes is always a string C-group, a smooth quotient of a Coxeter group with linear Coxeter-Dynkin diagram.

Using the vocabulary of incidence geometry, abstract regular polytopes can be defined as an incidence geometries which are thin, residually-connected, flag-transitive, and whose Buekenhout diagram is linear.
Contrary to the definition using ranked posets, in general there is no ordering o n the types of the elements in incidence geometries. This observation lead, in 2016, Fernandes, Leemans and Weiss ~\cite{hypertopes} to introduce the notion of \emph{regular hypertopes}. These are a thin, residually-connected, flag-transitive incidence geometries, with no restriction on the shape of their Buekenhout diagrams. The automorphism group of a regular hypertope is a \emph{C-group}, a smooth quotient of a Coxeter group. Its Coxeter-Dynkin diagram coincides with the Buekenhout diagram of the geometry.
The study of such structures has been a prolific area of research, with, for example, a classification of infinite families of both locally spherical \cite{fernandes2020exploration,piedade2023infinite,montero2022proper} and locally toroidal hypertopes \cite{fernandes2018hexagonal,fernandes2019two,ens2018rank,piedade2023infinite}. Additionally, classification of regular hypertopes of specified rank has been done for rank $3$ \cite{hou2019existence}, rank $4$ \cite{catalano2018hypertopes,ens2018rank}, and ranks $5$, $6$ and $7$ \cite{zhang2024abelian}.

In \cite{ARP}, McMullen and Schulte introduce different mixing operations on the automorphism group of abstract regular polytopes in order to construct new abstract regular polytopes out of existing ones. One operation of particular interest to us is the \emph{halving operation}. It is an operation $\eta$ defined originally on an abstract regular polyhedra $\Gamma$ of Schl\"{a}fli type $\{4,q\}$, with $q\geq 3$, such that $\eta(\Gamma)$ is an abstract regular polyhedra of Schl\"{a}fli type $\{q,q\}$. The reason this operation, as defined in \cite{ARP}, is restricted to regular polyhedra of Schl\"{a}fli type $\{4,q\}$ is that, if we try to apply this operation to polytopes of higher ranks, or rank $3$ polytopes of Schl\"{a}fli type different from $\{4,q\}$, the resulting C-group would have a non-linear Coxeter-Dynkin diagram. 
With the introduction of regular hypertopes, Montero and Weiss~\cite{MonteroWeiss} generalized the halving operation to any non-degenerate abstract regular polytope $\Gamma$, proving that $\eta(\Gamma)$ will always be a regular hypertope. 
The importance of the generalization of the halving operation introduced by \cite{MonteroWeiss} is that it allows, starting from a non-degenerate regular polytope, to obtain a regular hypertope with a Y-shaped diagram or a tail-triangle diagram. This makes this operation extremely useful for classifying new regular hypertopes. However, as this operation is defined on polytopes, it can only be applied once, since the resulting geometry will not have a linear diagram anymore. Moreover, the non-degenerate condition used in \cite{MonteroWeiss} is based on the works of Monson and Schulte\cite{MonsonSchulte}. In that context, a polytope is non-degenerate if its face poset is a lattice. This condition, although strong, is not easy to verify in general.

In \cite{Leemans2000}, Percsy, Percsy and Leemans show different, and quite general, ways to obtain new interesting incidence geometries from existing ones. These constructions are very concrete and do not depend on the existence of a well behaved automorphism group. 

Our goal in this article is to generalize the halving operation even further, so that it can be applied to regular hypertopes, under some necessary conditions. While working on this generalization, we realized that some of the constructions in \cite{Leemans2000}, when restricted to polytopes, produce geometries that have many things in common with the one obtained by the halving operation. As in \cite{Leemans2000}, all the geometries we consider will have to satisfy the following two conditions:
\begin{enumerate}
    \item[$(B_1)$] The $\{0,1\}$-truncation $\Gamma[0,1]$ of $\Gamma$ is the geometry of a simple graph \label{B1};
    \item[$(B_2)$] For a $1-$element $e$ and an $i-$element $i$ with $i \neq 0,1$, we have that $e * x$ if and only if $\sigma_0(e) \subset \sigma_0(x)$.\label{B2};
\end{enumerate}
where $\sigma_0(y)$ denotes the $0$-shadow of the elements $y \in \Gamma$ (see Section~\ref{subsec:incidence}).

In Section~\ref{sec:partitionedgeom}, we introduce a new construction, which is, in some sense, a generalization of the Construction 5.2 in \cite{Leemans2000}.

\begin{thmalpha}\label{thm:mainalpha}
    Let $\Gamma$ be a residually connected geometry of rank $n \geq 3$ over the type set $I = \{0,1,\cdots,n-1\}$ that satisfies $(B_1)$ and $(B_2)$ for the pair $\{0,1\}$, and such that the $\{0,1\}$-truncation of $\Gamma$ is not a bipartite graph. Then, there exists a geometry $\PP(\Gamma)(0,1)$ with diagram

    
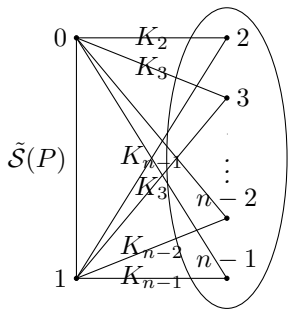
\begin{center}
         \begin{tikzpicture}[scale = 0.4]
    
         \filldraw[black] (2,-4) circle (2pt)  node[anchor=east]{$1$};
         \filldraw[black] (2,4) circle (2pt)  node[anchor=east]{$0$};
         \filldraw[black] (7,0.8) circle (0pt)  node[anchor=north]{$\vdots$};
         \filldraw[black] (7,2) circle (2pt)  node[anchor=west]{$3$};
         \filldraw[black] (7,-2) circle (2pt)  node[anchor=south]{$n-2$};
         \filldraw[black] (7,4) circle (2pt)  node[anchor=west]{$2$};
         \filldraw[black] (7,-4) circle (2pt)  node[anchor=south]{$n-1$};
         \draw (2,4) -- (2,-4)node [midway,left] (TextNode) {$\tilde{\mathcal{S}}(P)$};
         \draw (2,4) -- (7,2)node [midway] (TextNode) {$K_3$};
         \draw (2,-4) -- (7,2)node [midway] (TextNode) {$K_3$};
         \draw (2,4) -- (7,4)node [midway] (TextNode) {$K_2$};
         \draw (2,-4) -- (7,4)node [midway] (TextNode) {};
         \draw (2,4) -- (7,-2)node[midway] (TextNode) {};
         \draw (2,-4) -- (7,-2)node[midway] (TextNode) {$K_{n-2}$};
         \draw (2,4) -- (7,-4)node[midway] (TextNode) {$K_{n-1}$};
         \draw (2,-4) -- (7,-4)node[midway] (TextNode) {$K_{n-1}$};
         \draw (7,0) ellipse (2cm and 5cm);

    \end{tikzpicture}
    \end{center}

    where $\tilde{\mathcal{S}}(P)$ is the set $\{\tilde{\GG}(P) \mid \GG \in S, P \in \pi(\GG)\}$ (see Section~\ref{subsec:partitiongeom}) and the diagram of $I \setminus \{0,1\}$ is the same as for $\Gamma$.
    Furthermore, $\PP(\Gamma)$ is residually connected and $\PP(\Gamma)$ is, respectively, thin and flag-transitive if and only if $\Gamma$ is, respectively, thin and flag-transitive.
\end{thmalpha}

In Section~\ref{sec:halvinggeom}, we show that the algebraic halving operation can be constructed concretely as the halving geometry $H(\Gamma)$, as long as the starting geometry $\Gamma$ is a regular hypertope that satisfies the two conditions $(B_1)$ and $(B_2)$. This halving geometry is the geometry $\PP(\Gamma)$ of Theorem~\ref{thm:mainalpha} if the $\{0,1\}$-truncation of $\Gamma$ is not bipartite, and is the geometry $\BP(\Gamma)$, defined in \cite{Leemans2000} (see Theorem~\ref{thm:BP_fromDimi}), if the $\{0,1\}$-truncation of $\Gamma$ is bipartite. If $\Gamma$ is a regular hypertope, we have that  $G = \Aut(\Gamma) = \langle \rho_0, \rho_1, \cdots, \rho_{n-1} \rangle$ where $\rho_i$ is defined to be the unique automorphism sending $C$ to its unique $i$-adjacent chamber $C^i$. We can then define the halving group of $G$, denote by $H(G)$ to be $\langle \rho_0 \rho_1 \rho_0, \rho_1, \cdots, \rho_{n-1} \rangle$. We denote $H(G)_i$, for $i = 0,1,\cdots, n-1$, the subgroups of $H(G)$ generated by all the generators of $H(G)$, except from the $i$'th one. 

\begin{thmalpha}
    Let $\Gamma$ be a regular hypertope with $\Aut(\Gamma)=G$ and suppose that $\Gamma$ satisfies the conditions $(B_1)$ and $(B_2)$. Then, the coset geometry $(H(G),(H(G)_i)_{i \in I}))$ is isomorphic to $H(\Gamma)(0,1)$, where $H(G)$ is the halving group of $G$ and $H(\Gamma)$ is the halving geometry of $\Gamma$.
\end{thmalpha}

Finally in Section~\ref{sec:cubictoroid}, we give applications of the halving operation on cubic toroids, showing its usefulness in building new interesting geometries.

\section{Background}

\subsection{Incidence geometries}\label{subsec:incidence}
    A 4-tuple $\Gamma = (X,I,*,t)$ is called an \textit{incidence system} if
    \begin{enumerate}
        \item $X$ is a set whose elements are called the \textit{elements} of $\Gamma$,
        \item $I$ is the set of types of $\Gamma$, 
        \item $*$ is a symmetric and reflexive relation (called the \textit{incidence relation}) on $X$, and
        \item $t$ is a map from $X$ to $I$, called the \textit{type map} of $\Gamma$, such that distinct elements $x,y \in X$ with $x * y$ satisfy $t(x) \neq t(y)$. 
    \end{enumerate}
Elements of $t^{-1}(i)$ are called the elements of type $i$ (or $i$-elements) and the set of elements of type $i$ is denoted by $S_i(\Gamma) = S_i$.
The \textit{rank} of $\Gamma$ is the cardinality of the type set $I$.
In an incidence system $\Gamma$ a \textit{flag} is a set of pairwise incident elements. The type of a flag $F$, $t(F)$, that is the set of types of the elements of $F$. 
The \textit{cotype} of a flag $F$ is the complement in $I$ of the type of $F$.
A \textit{chamber} is a flag of type $I$. An incidence system $\Gamma$ is an \textit{incidence geometry} if all its maximal flags are chambers. The set of all chambers of an incidence system $\Gamma$ will be denoted as $\CC(\Gamma)$.

For any $J \subset I$, we denote by $\Gamma[J]$ the $J-$truncation of $\Gamma$. By abuse of notation, we will use $\Gamma[0,1]$ instead of $\Gamma[\{0,1\}]$ to denote the $\{0,1\}$-truncation of $\Gamma$. For any element $x \in \Gamma$, we denote by $\sigma_0(x)$ the $0$-shadow of $x$ (i.e: the set of $0$-elements incident to $x$).

Let $F$ be a flag of $\Gamma$. An element $x\in X$ is {\em incident} to $F$ if $x*y$ for all $y\in F$. The \textit{residue} of $\Gamma$ with respect to $F$, denoted by $\Gamma_F$, is the incidence system formed by all the elements of $\Gamma$ incident to $F$ but not in $F$. The \textit{rank} of the residue $\Gamma_F$ is equal to rank$(\Gamma) - |F|$. If $\Gamma$ is of rank $n$, then the \textit{corank} of a flag $F$ is the difference between the rank of $\Gamma$ and the rank of $F$.

The \textit{incidence graph} of $\Gamma$ is a graph with vertex set $X$ and where two elements $x$ and $y$ are connected by an edge if and only if $x * y$. 
An incidence geometry $\Gamma$ is \textit{connected} if its incidence graph is connected. It is \textit{residually connected} if all its residues of rank at least two are connected. It is \textit{thin} if all its residues of flags of corank one contain exactly two elements. Let $C$ be a chamber of a thin incidence geometry $\Gamma$. Then, there exists exactly one chamber of $\Gamma$ that differs from $C$ only by its element of type $i\in I$. We denote that chamber as $C^i$ and called it the \emph{$i$-adjacent chamber} of $C$.

Francis Buekenhout introduced in~\cite{buek} a new diagram associated to incidence geometries. His idea was to associate to each rank two residue a set of three integers giving information on its incidence graph.
Let $\Gamma$ be a rank $2$ geometry. We can consider $\Gamma$ to have type set $I = \{P,L\}$, standing for points and lines. The {\em point-diameter}, denoted by $d_P(\Gamma) = d_P$, is the largest integer $k$ such that there exists a point $p \in P$ and an element $x \in \Gamma$ such that $d(p,x) = k$, where $d(p,x)$ denotes the distance between $p$ and $x$ in the incidence graph of $\Gamma$. Similarly the {\em line-diameter}, denoted by $d_L(\Gamma) = d_L$, is the largest integer $k$ such that there exists a line $l \in L$ and an element $x \in \Gamma$ such that $d(l,x) = k$. Finally, the \textit{gonality} of $\Gamma$, denoted by $g(\Gamma) = g$ is half the length of the smallest circuit in the incidence graph of $\Gamma$.

If a rank $2$ geometry $\Gamma$ has $d_P = d_L = g = n$ for some natural number $n$, we say that it is a \textit{generalized $n$-gon}. Generalized $2$-gons are also called generalized digons. They are in some sense trivial geometries since all points are incident to all lines. Their incidence graphs are complete bipartite graphs. Generalized $3$-gons are projective planes.

Let $\Gamma$ be a geometry over $I$.  The \textit{Buekenhout diagram} (or diagram for short) $D$ for $\Gamma$ is a graph whose vertex set is $I$. Each edge $\{i,j\}$ is labeled with a collection $D_{ij}$ of rank $2$ geometries. We say that $\Gamma$ belongs to $D$ if every residue of rank $2$ of type $\{i,j\}$ of $\Gamma$ is one of those listed in $D_{ij}$ for every pair of $i \neq j \in I$. In most cases, we use conventions to turn a diagram $D$ into a labeled graph. The most common convention is to not draw an edge between two vertices $i$ and $j$ if all residues of type $\{i,j\}$ are generalized digons, and to label the edge $\{i,j\}$ by a natural integer $n$ if all residues of type $\{i,j\}$ are generalized $n$-gons. It is also common to omit the label when $n=3$ and to put a double edge when $n=4$.

Let $\Gamma = \Gamma(X,I,*,t)$ be an incidence geometry. A {\em correlation} of $\Gamma$ is a bijection $\phi$ of $X$ respecting the incidence relation $*$ and such that, for every $x,y \in X$, if $t(x) = t(y)$ then $t(\phi(x)) = t(\phi(y))$. If, moreover, $\phi$ fixes the types of every element (i.e $t(\phi(x)) = t(x)$ for all $x \in X$), then $\phi$ is said to be an {\em automorphism} of $\Gamma$. The group of all correlations of $\Gamma$ is denoted by $\Cor(\Gamma)$ and the automorphism group of $\Gamma$ is denoted by $\Aut(\Gamma)$. Remark that $\Aut(\Gamma)$ is a normal subgroup of $\Cor(\Gamma)$ since it is the kernel of the action of $\Cor(\Gamma)$ on $I$.

If $\Aut(\Gamma)$ is transitive on the set of chambers of $\Gamma$ then we say that $\Gamma$ is {\em flag-transitive}. If moreover, the stabilizer of a chamber in $\Aut(\Gamma)$ is reduced to the identity, we say that $\Gamma$ is {\em simply transitive} or {\em regular}.

A \textit{regular hypertope} is a thin, residually connected and flag-transitive incidence geometry.

Incidence geometries can be obtained from a group $G$ together with a set $(G_i)_{i \in I}$ of subgroups of $G$ as described in~\cite{Tits1957}. 
    The \emph{coset geometry} $\Gamma(G,(G_i)_{i \in I})$ is the incidence geometry over the type set $I$ where:
    \begin{enumerate}
        \item The elements of type $i \in I$ are right cosets of the form $G_i g$, $g \in G$.
        \item The incidence relation is given by non-empty intersection. More precisely, the element $G_i g$ is incident to the element $G_j  k$ if and only if $i\neq j$ and $G_i g \cap G_j k \neq \emptyset$.
    \end{enumerate}

\subsection{C-groups}

A \emph{C-group of rank $n$} is a pair $(G,S)$, where $G$ is a group and $S:=\{\rho_0,\ldots,\rho_{n-1}\}$ is a generating set of involutions of $G$ that satisfy the following condition, called the \emph{intersection property},
$$\forall I,J\subseteq \{0,\ldots,n-1\},\ \langle \rho_i\ | \ i\in I\rangle\ \cap\ \langle \rho_j\ | \ j\in J\rangle\ = \ \langle \rho_k\ |\ k \in I\cap J\rangle.$$
A C-group is said to be a \emph{string C-group} if its generating involutions can be ordered in such a way that, for all $i,j$ where $|i-j|\geq 2$, we have that $(\rho_i\rho_j)^2 = id$.
A subgroup of $G$ generated by all but one involution of $S$ is called a \emph{maximal parabolic subgroup} and is denoted as $$G_i := \langle \rho_j\ |\ j\in I\backslash\{i\}\rangle,$$ with $I:=\{0,\ldots, n-1\}$.
By abuse of notation, when the generating set $S$ is evident, we will say $G$ is a C-group.

One important result regarding regular hypertopes is that their automorphism groups are C-groups, as expressed in the following proposition.
\begin{prop}\label{prop:CosetGeometryfromAction}\cite[Theorem 4.1]{hypertopes}
    Let $G$ be a group acting simply transitively on a regular hypertope $\Gamma$ and let $C$ be a chamber of $\Gamma$. Then $\Aut(\Gamma) = (G, \{\rho_0,\rho_1, \cdots, \rho_{n-1}\})$, where $\rho_i$ is the unique involution sending $C$ to its $i$-adjacent chamber $C^i$, is a $C$-group.
\end{prop}

The \emph{Coxeter diagram} of a C-group $(G,S)$, denoted $D(G,S)$, is the graph whose nodes represent the elements of $S$ and an edge between the generators $\rho_i$ and $\rho_j$ has label $p_{i,j}:=o(\rho_i\rho_j)$, the order of $\rho_i\rho_j$. By convention, edges with label equal to 2 are not drawn, and, whenever an edge has label 3, its label is omitted. When $p_{i,j} = 4$, we omit the label and put a double edge instead of a single edge. For string C-groups, this diagram is linear, and must then be as in Figure~\ref{coxeterdiagram}.
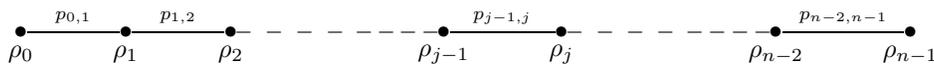
\begin{figure}[h]
 $$\xymatrix@-1.9pc{
*{\bullet} \ar@{-}[rrrrr]^{p_{0,1}} && & &&*{\bullet} \ar@{-}[rrrrr]^{p_{1,2}} && & &&*{\bullet} \ar@{--}[rrrrrr] && && && \ar@{--}[rrrr] && && *{\bullet} \ar@{-}[rrrrr]^{p_{j-1,j}} && & &&  *{\bullet} \ar@{--}[rrrrrr] && && && \ar@{--}[rrrr] && && *{\bullet} \ar@{-}[rrrrr]^{p_{n-2,n-1}} && & && *{\bullet} \\
*{\rho_0} && & &&*{\rho_1} && & && *{\rho_2} && && && && && *{\rho_{j-1}} && & && *{\rho_{j}} && && && && && *{\rho_{n-2}} && & && *{\rho_{n-1}}\\
}$$
 \caption{Coxeter Diagram of a string C-group.}
 \label{coxeterdiagram}
\end{figure}
If $G$ is the automorphism group of a regular hypertope $\Gamma$, the Coxeter diagram of $G$ coincides with the Buekenhout diagram of $\Gamma$.
Hence, if $\Gamma$ is a regular hypertope, we will denote by $D(\Gamma)$ the Coxeter diagram of its automorphism group. If $D(\Gamma)$ is linear, or, equivalently, if the automorphism group of $\Gamma$ is a string $C$-group, then $\Gamma$ is a \emph{(abstract) regular polytope}.
Conversely, if $G$ is a string C-group, the coset geometry $\Gamma(G,(G_i)_{i\in I})$ is a regular hypertope with linear diagram \cite[Theorem 5.1 and Theorem 5.2]{hypertopes}. Thus, regular polytopes are particular cases of regular hypertopes. We say a regular hypertope is \emph{proper} if it does not have a linear Coxeter diagram. Moreover, regular polytopes are in one-to-one correspondence with string C-groups.

We say a $C$-group $(G,S)$ is a \emph{Coxeter group} when its group relations are just the ones given by its Coxeter diagram $D(G,S)$. Hence, every $C$-group is either a Coxeter group or a smooth quotient of a Coxeter group. Given a $C$-group $G$, its universal covering group is the smallest Coxeter group $U$ such that there is a surjective homomorphism from $U$ to $G$. If $\Gamma$ is a geometry with automorphism group $G$, we will also say that $U$ is the universal covering group of $\Gamma$.

The \emph{Schl\"{a}fli type} of a regular polytope $\Gamma$ is defined as the set $\{p_{0,1},p_{1,2},\ldots,p_{n-2,n-1}\}$.
If a regular polytope has Schl\"{a}fli type $\{p_{0,1},p_{1,2},\ldots,p_{n-2,n-1}\}$, its universal covering group will be denoted as $[p_{0,1},p_{1,2},\ldots,p_{n-2,n-1}]$. It is the Coxeter group with relations $p_{0,1},p_{1,2},\ldots,p_{n-2,n-1}$.

In this paper, we will deal with Coxeter groups with various diagrams. We set notations for some of these diagrams and their associated hypertopes. Given a regular hypertope $\Gamma$, if $D(\Gamma)$ is 
$$
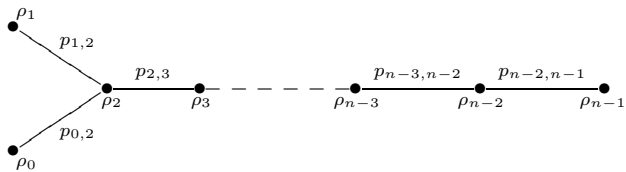
\xymatrix@-1.5pc{*{\bullet}\ar@{-}[rrrdd]^(0.01){\rho_1}^{p_{1,2}} \\
\\
 &&&*{\bullet}\ar@{-}[rrr]_(0.05){\rho_2}^{p_{2,3}}&&&*{\bullet}\ar@{--}[rrrrr]_(0.01){\rho_3}&&&&&*{\bullet}\ar@{-}[rrrr]_(0.01){\rho_{n-3}}^{p_{n-3,n-2}}&&&&*{\bullet}\ar@{-}[rrrr]_(0.01){\rho_{n-2}}_(0.99){\rho_{n-1}}^{p_{n-2,n-1}}&&&&*{\bullet}\\
 \\
 *{\bullet}\ar@{-}[rrruu]_(0.01){\rho_0}_{p_{0,2}}\\
 } $$
we will say that $\Gamma$ is of Schl\"{a}fli type $\left\{\nfrac{p_{1,2}}{p_{0,2}},p_{2,3},\ldots,p_{n-2,n-1}\right\}$, and its universal covering group will be denoted as $\left[\nfrac{p_{1,2}}{p_{0,2}},p_{2,3},\ldots,p_{n-2,n-1}\right]$. Similarly, if a $D(\Gamma)$ is 
 $$
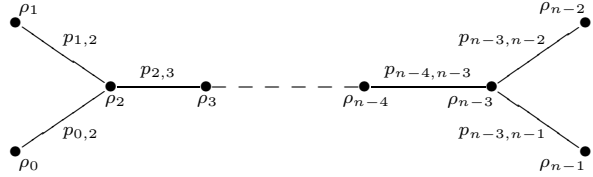
\xymatrix@-1.5pc{
*{\bullet}\ar@{-}[rrrdd]^(0.01){\rho_1}^{p_{1,2}} &&&&&&&&&&&&& &&& && *{\bullet} \\
\\
 &&&*{\bullet}\ar@{-}[rrr]_(0.05){\rho_2}^{p_{2,3}}&&&*{\bullet}\ar@{--}[rrrrr]_(0.01){\rho_3}&&&&&*{\bullet}\ar@{-}[rrrr]_(0.01){\rho_{n-4}}^{p_{n-4,n-3}}&&&&*{\bullet}\ar@{-}[rrruu]^(0.99){\rho_{n-2}}^{p_{n-3,n-2}}\ar@{-}[rrrdd]_(0.01){\rho_{n-3}}_(0.99){\rho_{n-1}}_{p_{n-3,n-1}}\\
 \\
 *{\bullet}\ar@{-}[rrruu]_(0.01){\rho_0}_{p_{0,2}}
 &&&&&&&&&&&&& &&& && *{\bullet} \\
 } $$
we will denote the Schl\"{a}fli type of $\Gamma$ as $\left\{\nfrac{p_{1,2}}{p_{0,2}},p_{2,3},\ldots,p_{n-4,n-3},\nfrac{p_{n-3,n-2}}{p_{n-3,n-1}}\right\}$, and its universal covering group as $\left[\nfrac{p_{1,2}}{p_{0,2}},p_{2,3},\ldots,p_{n-4,n-3},\nfrac{p_{n-3,n-2}}{p_{n-3,n-1}}\right]$.
Finally, if $D(\Gamma)$ is either
 $$
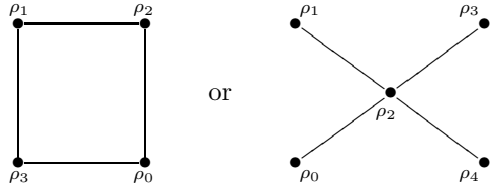
\xymatrix@-1.5pc{
*{\bullet}\ar@{-}[dddd]\ar@{-}[rrrr]^(0.01){\rho_1}^(0.99){\rho_2} &&&& *{\bullet}\ar@{-}[dddd] &&&& *{\bullet}\ar@{-}[rrrdd]^(0.01){\rho_1} & &&& && *{\bullet} \\
\\
&&&&&& \textnormal{or}&&&&&*{\bullet}\ar@{-}[rrruu]^(0.99){\rho_{3}}\ar@{-}[rrrdd]_(0.1){\rho_{2}}_(0.99){\rho_{4}} \\
 \\
*{\bullet}\ar@{-}[rrrr]_(0.01){\rho_3}_(0.99){\rho_0} &&&& *{\bullet}  &&&&*{\bullet}\ar@{-}[rrruu]_(0.01){\rho_0}
 & &&& && *{\bullet} \\
 } $$
we will denote their Schl\"{a}fli type as $\{(3,3,3,3)\}$ and $\{3^{1,1,1,1}\}$, respectively, with their respective universal covering groups being denoted by $[(3,3,3,3)]$ and $[3^{1,1,1,1}]$, respectively.

\subsection{Halving operation and Halving group}

The \emph{halving operation} $\eta$, as described in \cite{ARP}, applies only to a regular polyhedron $Q$ of Schl\"{a}fli type $\{4, q\}$ for some $q \geq 3$. It turns $Q$ into a polyhedron $P$ of Schl\"{a}fli type $\{q, q\}$.
This operation, acting on the automorphism group of $Q$, is defined as follows
$$\eta : (\rho_0, \rho_1, \rho_2) \rightarrow (\rho_0\rho_1\rho_0, \rho_2, \rho_1) =: (\varrho_0, \varrho_1, \varrho_2),$$
In other words, the first generator of the automorphism group of $Q$ is substituted by $\rho_0\rho_1\rho_0$, and the other two are swapped.
The generators $\rho_2$ and $\rho_1$ are swapped in order to maintain the linearity of the diagram. Recall that there is a $1$ to $1$ correspondence from string $C$-groups to abstract polytopes, so that this operation, while defined on the groups, can be seen as an operation on polytopes also.
In this paper we will deal with regular hypertopes, where its automorphism group is a C-group not necessarily string. Hence, the ordering of the generators will not be important to us.

Recently~\cite{MonteroWeiss}, the halving operation was revisited and generalized. Let $n\geq 3$ and let $\Gamma$ be a regular non-degenerate polytope of Schl\"{a}fli type $\{p_1,\ldots,p_{n-1}\}$ with automorphism group $G = \Aut(\Gamma) = \langle \rho_0,\ldots,\rho_{n-1}\rangle$. Here, a polytope is non-degenerate if its face poset is a lattice.
The halving operation is then the map 
$$\eta : \langle \rho_0,\rho_1,\rho_2,\ldots,\rho_{n-1}\rangle \rightarrow \langle \rho_0\rho_1\rho_0,\rho_1,\rho_2,\ldots,\rho_{n-1}\rangle =: \langle \tilde{\rho_0},\rho_1,\rho_2,\ldots,\rho_{n-1}\rangle.$$

The \emph{halving group} of $G$ is the image of the halving operation on $\Aut(\Gamma)=G$, being denoted by $H(G)$.
The main result of ~\cite{MonteroWeiss} is that, if $n\geq 3$ and $\Gamma$ is a non-degenerate polytope with $\Aut(\Gamma)=G$, then $H(G) = \langle \tilde{\rho_0}, \rho_1,\ldots,\rho_{n-1}\rangle$ is a C-group \cite[Theorem 3.1]{MonteroWeiss} with the following diagram

$$
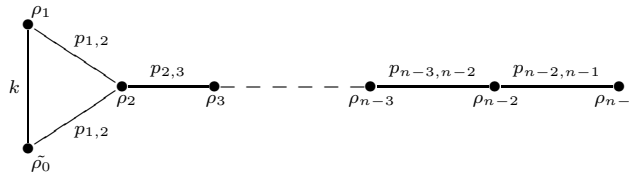
\xymatrix@-1.5pc{*{\bullet}\ar@{-}[rrrdd]^(0.01){\rho_1}^{p_{1,2}} \\
\\
 &&&*{\bullet}\ar@{-}[rrr]_(0.05){\rho_2}^{p_{2,3}}&&&*{\bullet}\ar@{--}[rrrrr]_(0.01){\rho_3}&&&&&*{\bullet}\ar@{-}[rrrr]_(0.01){\rho_{n-3}}^{p_{n-3,n-2}}&&&&*{\bullet}\ar@{-}[rrrr]_(0.01){\rho_{n-2}}_(0.99){\rho_{n-1}}^{p_{n-2,n-1}}&&&&*{\bullet}\\
 \\
 *{\bullet}\ar@{-}[rrruu]_(0.01){\tilde{\rho_0}}_{p_{1,2}}\ar@{-}[uuuu]^k\\
 } $$
where $k=p_{0,1}$ if $p_{0,1}$ is odd, and $k=\frac{p_{0,1}}{2}$ otherwise. Moreover, the coset geometry of the halving group is a regular hypertope \cite[Corollary 3.2]{MonteroWeiss}. 

The non-degeneracy condition used in \cite{MonteroWeiss} requires the abstract regular polytope $\Gamma$ to be a lattice ~\cite{danzer_regular_1984,ARP,schulte_regular_1985}.
One of the consequences of this non-degeneracy condition on polytopes is that all faces of type $i$ can be considered as a subset of the set of vertices of the polytope $\Gamma$. In addition, given $i\leq j$, every $j$-face of a non-degenerate polytope $\Gamma$ is determined by the set of $i$-faces incident to it \cite{danzer_regular_1984}. In this paper, we will consider weaker conditions to impose on a regular hypertope so that, after the halving operation, the structure is still a regular hypertope.

\section{The partitioned geometry}\label{sec:partitionedgeom}

Let $\Gamma$ be an incidence geometry over the type set $I = \{0,1,\cdots,n-1\}$. Let $\{k,j\}\subset \{0,1,\ldots,n-1\}$ be a subset of $I$ and, for ease of notation, let $\{k,j\}=\{0,1\}$.
Suppose $\Gamma$ satisfies the following conditions:
\begin{enumerate}
    \item[$(B_1)$] The $\{0,1\}$-truncation $\Gamma[0,1]$ of $\Gamma$ is the geometry of a simple graph \label{B1};
    \item[$(B_2)$] For a $1-$element $e$ and an $i-$element $i$ with $i \neq 0,1$, we have that $e * x$ if and only if $\sigma_0(e) \subset \sigma_0(x)$.\label{B2}
\end{enumerate}
As shown in \cite{Leemans2000}, this implies that the diagram of $\Gamma$ is as in Figure \ref{fig:leafDiagram}, where $\mathcal{S}$ is a collection of graphs, $K_2,\cdots, K_{n-1}$ are arbitrary classes of rank $2$ geometries, and where the diagram on $I \setminus \{0,1\}$ is arbitrary.

 Given two vertices $p,q\in\Gamma$ of a graph $\GG$, we use $p \sim q$, and say that $p$ and $q$ are \textit{adjacent}, to mean that the vertices $p$ and $q$ are both incident to some common edge in $\GG$. Whenever a geometry $\Gamma$ satisfies $(B_1)$, we will consider its $\{0,1\}$-truncation $\Gamma[0,1]$ as a graph, and call the $0$ and $1$ elements of $\Gamma$ the vertices and edges of $\Gamma$.
\begin{figure}
\begin{center}
    \begin{tikzpicture}[scale = 0.5]
    
   \filldraw[black] (-2,0) circle (2pt)  node[anchor=north]{$0$};
   \filldraw[black] (2,0) circle (2pt)  node[anchor=north]{$1$};
   \filldraw[black] (7,0.8) circle (0pt)  node[anchor=north]{$\vdots$};
    \filldraw[black] (7,2) circle (2pt)  node[anchor=west]{$3$};
    \filldraw[black] (7,-2) circle (2pt)  node[anchor=south]{$n-2$};
    \filldraw[black] (7,4) circle (2pt)  node[anchor=west]{$2$};
    \filldraw[black] (7,-4) circle (2pt)  node[anchor=south]{$n-1$};
    \draw (2,0) -- (-2,0)node [midway,above] (TextNode) {$\mathcal{S}$};
    \draw (2,0) -- (7,2)node [midway] (TextNode) {$K_3$};
    \draw (2,0) -- (7,4)node [midway] (TextNode) {$K_2$};
    \draw (2,0) -- (7,-2)node[midway] (TextNode) {$K_{n-2}$};
    \draw (2,0) -- (7,-4)node[midway] (TextNode) {$K_{n-1}$};
    \draw (7,0) ellipse (2cm and 5cm);

    \end{tikzpicture}
    \end{center}
    \caption{The diagram of a geometry that satisfies $(B_1)$ and $(B_2)$.}
    \label{fig:leafDiagram}
 \end{figure}
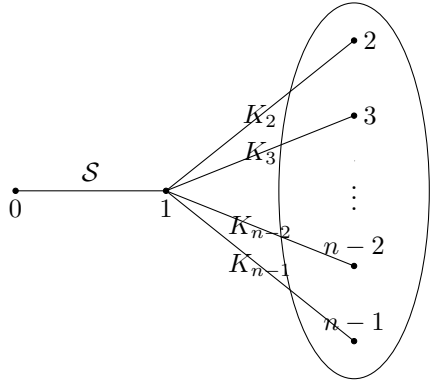

In \cite{Leemans2000}, the authors introduce the construction of a geometry, that we will denote by $\BP(\Gamma)(0,1)$. Here, the starting geometry $\Gamma$ must satify both $(B_1)$ and $(B_2)$, and $\Gamma[0,1]$ must be bipartite. 

\begin{construction}\label{BPRH}\cite[Construction 6.2]{Leemans2000}
Let $\Gamma := (\cup_{i\in I}^{n-1}S_i,I,t,*)$ be a geometry of (possibly infinite) rank $n\geq 3$ satisfying $(B_1)$ and $(B_2)$ for the pair $\{0,1\}$, with the diagram of Figure~\ref{fig:leafDiagram}, where $\mathcal{S}$ is a class of graphs, $K_i$ are arbitrary classes of rank two geometries and the diagram on $I\backslash\{0,1\}$ is arbitrary. Moreover let the $\{0,1\}$-truncation be the geometry of bipartite graph with partition $\{X_0,X_1\}$ on the vertex set of the $\{0,1\}$-truncation. We define a geometry $\BP(\Gamma)(0,1):=(\cup_{i=0}^{n-1}X_i,I',t',*')$ such that:
\begin{itemize}
    \item \begin{equation*}
    X_i=\begin{cases}
    S_i, & \text{if $i\notin\{0,1\}$}.\\
    X_i, & \text{if $i\in\{0,1\}$}.
  \end{cases}
\end{equation*}
    \item $I' = I$
    \item $t'(X_i) = i$
    \item for $x\in X_i$ and $y\in X_j$, with $i\neq j$ then
        \begin{itemize}
            \item if $\{i,j\}\neq\{0,1\}$, then $x*'y$ if and only if $x* y$
            \item if $\{i,j\}=\{0,1\}$, then $x*' y$ if and only if $x\thicksim y$ in the $\{0,1\}$-truncation.
        \end{itemize}
\end{itemize}
\end{construction}

In the same paper, the authors introduce the construction of a geometry $\tilde{\Gamma}(0,1)$, starting from a geometry $\Gamma$ satisfying both $(B_1)$ and $(B_2)$, with $\Gamma[0,1]$ being non-bipartite. 

\begin{construction}\label{NBPRH}\cite[Construction 5.2]{Leemans2000}
Let $\Gamma := (\cup_{i\in I}^{n-1}S_i,I,t,*)$ be a geometry of (possibly infinite) rank $n\geq 3$ satisfying $(B_1)$ and $(B_2)$ for the pair $\{0,1\}$, with the diagram of Figure~\ref{fig:leafDiagram}, where $\mathcal{S}$ is a class of graphs, $K_i$ are arbitrary classes of rank two geometries and the diagram on $I\backslash\{0,1\}$ is arbitrary. Moreover let the $\{0,1\}$-truncation be the geometry of a non-bipartite graph. Let $X_0 := \{(v,0):v\in S_0\}$ and $X_1 := \{(v,1):v\in S_0\}$.  We define a geometry $\tilde{\Gamma}(0,1):=(\cup_{i=0}^{n-1}X_i,I',t',*')$ such that:
\begin{itemize}
    \item \begin{equation*}
    X_i=\begin{cases}
    S_i, & \text{if $i\notin\{0,1\}$}.\\
    X_i, & \text{if $i\in\{0,1\}$}.
  \end{cases}
\end{equation*}
    \item $I' = I$
    \item $t'(X_i) = i$
    \item for $x\in X_i$ and $y\in X_j$, with $i\neq j$ then
        \begin{itemize}
            \item if $\{i,j\}\neq\{0,1\}$, then $x*'y$ if and only if $x* y$
            \item if $\{i,j\}=\{0,1\}$, then, as $x=(v_x,i)$ and $y=(v_y,j)$, $x*' y$ if and only if $v_x\thicksim v_y$ in the $\{0,1\}$-truncation of $\Gamma$.
        \end{itemize}
\end{itemize}
\end{construction}

The main problem is that Construction~\ref{NBPRH} fails to be residually connected if the $\{0,1\}$-truncation of any of its residues is a bipartite graph (see \cite[Theorem 4.1]{Leemans2000}). Our first task is to address that issue.

In what follows, we will introduce a new construction which extends Construction~\ref{NBPRH}. This new construction is isomorphic to the one of Construction~\ref{NBPRH} when the $\{0,1\}$-truncation of all residues of $\Gamma$ are non-bipartite.

\subsection{Partitioned Neighbourhood Geometry}\hfill \label{subsec:partitiongeom}
\\

Let $\GG$ be a connected graph with vertex set $V$. We define an equivalence relation on the vertices of $\GG$ by setting, for two vertices $p$ and $q$, that $p \equiv q$ if and only if there exists a path of even length from $p$ to $q$ in $\GG$. Let $\pi(\GG)$ be the set of equivalence classes of $\GG$ under that relation. It is straightforward to check that $|\pi(\GG)| = 1$ if $\GG$ contains an odd cycle and $|\pi(\GG)|= 2$ otherwise. Note that not having any cycle of odd length is the same as being bipartite. In that case, the equivalence classes correspond to the proper $2$-colorings of $\GG$. For $P \in \pi(\GG)$, we define $\Bar{P}$ to be $V \setminus P$ if $|\pi(\GG)|= 2$ and $\Bar{P} = P$ if $|\pi(\GG)|= 1$.

\begin{lemma}\label{lem:adjacent}
    If $p$ and $q$ are adjacent vertices of $\GG$ and $P \in \pi(\GG)$, then $p \in P$ if and only if $q \in \Bar{P}$.
\end{lemma}

\begin{proof}
    If $\GG$ is not bipartite, $P = \Bar{P} = \GG$ and the lemma holds trivially. If $\GG$ is bipartite, there exists a proper $2$-coloring of $\GG$ such that $P$ corresponds to the sets of vertices of $\GG$ of a given color. The set $\Bar{P}$ then corresponds to the vertices of the other color. Since two adjacent vertices cannot have the same color, the lemma holds.
\end{proof}

Let $P \in \pi(\GG)$ and let $V$ and $E$ be the vertex and edge set of the connected graph $\GG$. The \textit{partitioned neighborhood geometry} $\Tilde{\GG}$ of $\GG$ with respect to $P$ is the rank two geometry $\Tilde{\GG}(P) = (P \times \{0\} \cup \Bar{P} \times \{1\}, \{0,1\}, \Tilde{t},\tilde{*})$ where 
\begin{itemize}
    \item $\Tilde{t}(p,i) = i$ for any $p \in V$ and $i =0,1$.
    \item $(p,0) \Tilde{*} (q,1)$ if and only if $\{p,q\} \in E$.
\end{itemize}

In particular, when $\GG$ is not bipartite, and hence $P = V$, we recover the more classical definition of neighborhood geometry, as defined in \cite{Leemans2000}.

\begin{lemma}\label{lem:PnbghGeom}
    $\Tilde{\GG}(P)$ is connected if and only if $\GG$ is connected
\end{lemma}

\begin{proof}
    There exists a path from $(p,i)$ to $(q,i)$ in $\Tilde{\GG}(P)$ if and only if there is a path of even length from $p$ to $q$ in $\GG$. 
    By definition of $\Tilde{\GG}(P)$, if $(p,i)$ and $(q,i)$ are in $\Tilde{\GG}(P)$, we have that $p \equiv q$. But $p \equiv q$ if and only if there is a path of even length between $p$ and $q$. Similarly, there exists a path from $(p,i)$ to $(q,i + 1 \pmod 2)$ in $\Tilde{\GG}(P)$ if and only if there is a path of odd length from $p$ to $q$ in $\GG$ and $(p,i)$ and $(q,i+ 1 \pmod 2)$ are in $\Tilde{\GG}(P)$ if and only if $p$ is not in the same equivalence class as $q$.
\end{proof}

The gonality $g$ of $\GG$ is the smallest number such that there exists a circuit of length $2g$ in the incidence graph of $\GG$.

\begin{lemma}\label{lem:gonal}
    Let $g$ be the gonality of $\GG$. If $g$ is even, then the gonality $\tilde{g}$ of $\tilde{\GG}(P)$ is $g/2$. If $g$ is odd then the gonality $\tilde{g}$ of $\tilde{\GG}(P)$ is:
    \begin{itemize}
        \item either $\tilde{g}=g$, if $\GG$ has no circuit of even length smaller than $2g$;
        \item or $\tilde{g}=k/2$, if $\GG$ admits a circuit of even length $k<2g$ and $k$ is minimal for that property (with $g<k$).
    \end{itemize}
\end{lemma}
\begin{proof}
    Let $\GG$ be a graph and let $\pi(\GG)$ be the set of equivalence classes of $\GG$ as defined above.
    If $|\pi(\GG)|=1$, we have that the partitioned neighbourhood geometry $\tilde{\GG}(P)$ is equivalent to the neighbourhood geometry of \cite{Leemans2000}. Hence, for this case the results follow from Lemma 3.6 of \cite{Leemans2000}.
    If $|\pi(\GG)|=2$, then $\GG$ is bipartite and, hence, all circuits in $\GG$ are of even length, hence $g$ is even. By taking $P\in\pi(\GG)$, we have that the smallest circuit of the incidence graph of $\tilde{\GG}(P)$ has length $g$, hence $\tilde{g}=g/2$.
\end{proof}

Note that, in $\Gamma$, condition ($B_1$) guarantees that the $\{0,1\}-$truncations of any residues of cotype containing $\{0,1\}$ are simple graphs. Hence, it makes sense to talk about their partitioned neighborhood geometries. In particular, for any element $x \in \Gamma$ of type $i \neq 0,1$, we can consider the set $\pi(\Gamma_x[0,1])$. For the rest of the article, we will call the $0$-elements of $\Gamma$ the \textit{vertices} of $\Gamma$ and the $1$-elements of $\Gamma$ the \textit{edges} of $\Gamma$. By $(B_1)$, any edge $e \in \Gamma$ has exactly two incident vertices $p$ and $q$. Moreover, by $(B_1)$, any two vertices $p$ and $q$ either define exactly one edge $e=\{p,q\}$ or none at all (i.e. the edges are uniquely vertex-describable). We will then write $e = \{p,q\}$.

\subsection{Partitioned Geometry}\hfill
\\

We now introduce a new construction, called the \textit{partitioned geometry} of $\Gamma$, for residually connected geometries $\Gamma$ that satisfy $(B_1)$ and $(B_2)$. It is a variant of Construction 5.1 in \cite{Leemans2000}. The organization and proofs of this section loosely follow the ones in \cite{Leemans2000}.

Let $\Gamma$ be a incidence geometry with diagram as in Figure \ref{fig:leafDiagram}. Then, by condition $(B_1)$, for any $x \in \Gamma$, $\Gamma_x[0,1]$, the $\{0,1\}-$truncation of the residue of $x$, is the geometry of a simple graph. If, moreover, $\Gamma$ is residually connected, $\Gamma_x[0,1]$ must be the geometry of a connected graph.

\begin{construction} \label{constr:notBipartite}
    Let $\Gamma = (\cup_{i=0}^{n-1}S_i,I,t,*)$ be a residually connected geometry that satisfies $(B_1)$ and $(B_2)$, and such that $\Gamma[0,1]$ is not bipartite. The partitioned geometry $\PP(\Gamma)(0,1) = (\cup_{i=0}^{n-1}X_i,I',t',*')$ is defined as follows
    \begin{itemize}
        \item $X_i = \{(x,P_x) \mid x \in S_i,P_x \in \pi(\Gamma_{x}[0,1])\}$ for $i \neq 0,1$ and $X_i = S_0 \times \{i\} $ for $i = 0,1$
        \item $I' = I$
        \item $t'(X_i) = i$
        \item given $\xi \in X_i$ and $\psi \in X_j$, with $i \neq j$, we define
        \begin{itemize}
            \item if $i,j \notin \{0,1\}$, we have $\xi =(x,P_x) *' (y,P_y) = \psi$ if and only if $x * y$ and $P_x \cap P_y \neq \emptyset$
            \item if $i = 0$ and $j \neq 1$, we have that $\xi=(p,0) *'(y,P_y)=\psi$ if and only if $p*y$ and $p \in P_y$
            \item if $i = 1$ and $j \neq 0$, we have that $\xi=(q,1) *'(y,P_y)=\psi$ if and only if $q*y$ and $q \in \Bar{P_y}$
            \item if $i,j \in\{0,1\}$, we have that $\xi = (p,i) *'(q,j) =\psi$ if and only if $p \sim q$ in $\Gamma[0,1]$, seen as a graph. 
        \end{itemize}
    \end{itemize}
\end{construction}

For lightening the notations, whenever there is no risk of confusion, we will write $\PP(\Gamma)$ instead of $\PP(\Gamma)(0,1)$. We start by a technical lemma, that shows that the partitions of pairwise incident elements must always all intersect non-trivially.

\begin{lemma}\label{lem:tripleintersection}
    Let $x,y,z$ be three elements of $\Gamma$ of different types, not equal to $0$ or $1$, and let $\{(x,P_x),(y,P_y),(z,P_z)\}$ be a flag of $\PP(\Gamma)$. Then $P_x \cap P_y \cap P_z$ is non empty.
\end{lemma}

\begin{proof}
    If any of $\Gamma_x[0,1], \Gamma_y[0,1]$ or $\Gamma_z[0,1]$ is not bipartite, the lemma holds trivially. We can thus suppose that they are all bipartite. $\Gamma$ is a geometry and $\{x,y,z\}$ is a flag in $\Gamma$. Hence, there exists a vertex $p$ of $\Gamma$ such that $p$ is incident to all three of $x,y$ and $z$. Without loss of generality, suppose that $p \in P_x$. We claim that $p \in P_x \cap P_y \cap P_z$. By definition of incidence in $\PP(\Gamma)$, we have that $P_x \cap P_y$ is non empty. Hence, there exists a vertex $q \in P_x \cap P_y$. This means that there is a bipartition of $\Gamma_x[0,1]$ in which $p$ and $q$ are in the same component. Since $\Gamma$ is residually connected, the graph $\Gamma_x[0,1] \cap \Gamma_y[0,1]$ is then also connected and bipartite, with $p$ and $q$ still in the same component. The vertices $p$ and $q$ must then also be in the same partition of $\Gamma_y[0,1]$. Hence, $p \in P_y$. We get in the exact same way that $p \in P_z$.
\end{proof}

\begin{thm}\label{PGammageom}
    The incidence system $\PP(\Gamma)$ is a geometry.
\end{thm}

\begin{proof}
    Let $\Phi$ be a flag of $\PP(\Gamma)$ and let $J = t(\Phi)$. We show that $\Phi$ can always be extended to a chamber.
    \begin{itemize}
        \item If $J \subset I \setminus \{0,1\}$, then $\Phi = \{(x_j,P_{x_j}) \mid j \in J\}$. We can consider the flag $F = \{ x_j \mid j \in J\}$ of $\Gamma$. Extend $F$ to a chamber $C = \{p,e,x_2,\cdots, x_{n-1}\}$ such that $p\in \cap_{j\in J}P_{x_j}$. Note that $\cap_{j\in J}P_{x_j}$ is non empty by Lemma~\ref{lem:tripleintersection}.  By $(B_1)$, $e = \{p,q\}$ for some $q \in \cap_{j\in J}\Bar{P}_{x_j}$. Then $\Theta = \{(p,0),(q,1),(x_2,P_{x_2}), \cdots, (x_{n-1},P_{x_{n-1}}) \}$ is a chamber of $H(\Gamma)$, where $p \in P_{x_i}$ for every $i\in I$. By Lemma~\ref{lem:adjacent}, we have that $q \in \Bar{P}_{x_i}$ for all $i$. Moreover $P_{x_i} \cap P_{x_j}$ always contains $p$, and is therefore not empty. Hence $\Theta$ is indeed a chamber.
        \item If $0 \in J$ but $1 \notin J$, then $\Phi = \{(p,0),(x_j,P_{x_j}) \mid j \in J \setminus \{0\} \}$. Consider the flag $F = \{p,x_j \mid j \in J \setminus \{0\} \}$ of $\Gamma$. It can be extended to a chamber $C = \{p,e,x_2,\cdots, x_{n-1}\}$, with $e=\{p,q\}$. As in the previous case, this yields a chamber $\Theta = \{(p,0),(q,1),(x_2,P_{x_2}), \cdots, (x_{n-1},P_{x_{n-1}}) \}$ of $\PP(\Gamma)$ containing $\Phi$. Note that $(B_2)$ guarantees that $x_i * q$, since $x_i * e$ for all $i \in J$.
        \item If $1 \in J$ but $0 \notin J$, then $\Phi = \{(q,1),(x_j,P_{x_j}) \mid j \in J \setminus \{1\} \}$. Consider the flag $F = \{q,x_j \mid j \in J \setminus \{1\} \}$ of $\Gamma$. It can be extended to a chamber $C = \{q,e,x_2,\cdots, x_{n-1}\}$, where $e = \{p,q\}$ is any edge containing $q$. Then, once again, $\Theta = \{(p,0),(q,1),(x_2,P_{x_2}), \cdots, (x_{n-1},P_{x_{n-1}}) \}$ is a chamber of $\PP(\Gamma)$ containing $\Phi$. Indeed, by $(B_2)$, since $x_i * e$ for all $i$, we also have that $x_i * p$ for all $i$.
        \item Finally, if $0,1 \in J$, we have that $\Phi = \{(p,0),(q,1),(x_j,P_{x_j}) \mid j \in J \setminus \{0,1\} \}$. Let $e$ be the unique edge of $\Gamma$ incident to both $p$ and $q$. Then, by $(B_2)$, $F = \{p,e,x_j \mid j \in J \setminus \{0,1\} \}$ is a flag of $\Gamma$. Extend it to a chamber $C = \{p,e,x_2,\cdots, x_{n-1}\}$. The chamber $\Theta = \{(p,0),(q,1),(x_2,P_{x_2}), \cdots, (x_{n-1},P_{x_{n-1}}) \}$, where $p \in P_{x_i}$ for every $i\in I$, is a chamber of $\PP(\Gamma)$ containing $\Phi$.
    \end{itemize}
\end{proof}

 In Sections~\ref{subsec:PP_residues}-\ref{subsec:PPb1b2}, we will prove that when $\Gamma[0,1]$ is not bipartite, $\PP(\Gamma)$ is residually connected. We also show that $\PP(\Gamma)$ is thin if and only if $\Gamma$ is thin and we investigate the automorphisms of $\PP(\Gamma)$.

\subsection{Residues of $\PP(\Gamma)$}\label{subsec:PP_residues}\hfill
\\

We now study the isomorphism classes of residues of $\PP(\Gamma)$ and see that they can all be deduced from the ones of $\Gamma$.

\begin{lemma}\label{lem:residue0}
    Let $\chi = (p,i) \in S_i$ where $i = 0,1$. Then, there is an isomorphism from $\PP(\Gamma)_\chi$ to $\Gamma_p$.
\end{lemma}

\begin{proof}
    Suppose $i =0$. Let $\psi$ be an element of $\PP(\Gamma)$ incident to $\chi$. If $t(\psi) \neq 1$, we have that $\psi = (y,P_y)$ and we can define $\alpha(\psi) = y$. Note that $P_y$ is uniquely defined by the fact that $p \in P_y$. If $t(\psi) = 1$, we have that $\psi = (q,1)$ and we define $\alpha(\psi) = e$ where $e$ is the unique $1$-element incident to both $p$ and $q$. It is easy to check to $\alpha$ is both bijective and incidence preserving. The proof for $i = 1$ is identical, defining $P_y$ such that $p\in\Bar{P}_y$ when $t(\psi) \neq 0$, and $\psi=(q,0)$ when $t(\psi) = 0$.
\end{proof}

\begin{lemma}\label{lem:residue0J}
    Let $\Phi$ be a flag of type $J$ of $\PP(\Gamma)$ with $0 \in J$ or $1 \in J$. Then $\PP(\Gamma)_{\Phi}$ is isomorphic to $\Gamma_F$ for some flag $F$ of $\Gamma$.
\end{lemma}

\begin{proof}
    Suppose first that $0 \in J$ but $1 \notin J$. Then, $\Phi = \{(p,0),(x_j,P_{x_j}) \mid j \in J \setminus \{0\}\}$ and $F = \{p,x_j \mid j \in J \setminus \{0\}\}$ is a flag of $\Gamma$. The isomorphism from Lemma~\ref{lem:residue0} restricts to an isomorphism from $\PP(\Gamma)_{\Phi}$ to $\Gamma_F$. The same holds for $1 \in J$ but $0 \notin J$. Finally, suppose that $\{0,1\} \subseteq J$. Then, $\Phi = \{(p,0),(q,1),(x_j,P_{x_j}) \mid j \in J \setminus \{0,1\}\}$. Let $e = \{p,q\}$ be the only edge between $p$ and $q$ in $\Gamma$. Then, by $(B_2)$, $F = \{p,e,x_j \mid j \in J \setminus \{0,1\}\}$ is a flag of $\Gamma$. It is straightforward to check that $\PP(\Gamma)_{\Phi}$ is isomorphic to $\Gamma_F$.
\end{proof}

Let $\Phi$ be any flag of $\PP(\Gamma)$ of type $J \subseteq (I \setminus \{0,1\})$. Then $\Phi = \{(x_j,P_{x_j}) \mid j \in J$\}. We say that the flag $F = \{x_j \mid j \in J\}$ is the flag obtained from $\Phi$ by forgetting the partitions.

\begin{lemma}\label{lem:residue01}
    Let $\Phi = \{(x_2,P_{x_2}), \cdots, (x_{n-1},P_{x_{n-1}})\}$ be a flag of $\PP(\Gamma)$ of cotype $\{0,1\}$ and let $F = \{x_2,\cdots,x_{n-1}\}$ be the flag of $\Gamma$ obtained by forgetting the partitions. Then, $\PP(\Gamma)_{\Phi}$ is isomorphic to $\Tilde{\GG}(\cap_{i=2}^{n-1} P_{x_i})$, where $\GG$ is $\Gamma_F$, considered as a graph. 
\end{lemma}

\begin{proof}
    Let $\chi$ be an element of $\PP(\Gamma)_{\Phi}$. Then $\chi = (p,i)$ for $i = 0,1$ with $p \in \Gamma_F$. Moreover, if $i =0$, we have that $p \in P_{x_j}$ for all $j\in I \setminus \{0,1\}$, and if $i = 1$, we have that $p \in \Bar{P}_{x_j}$ for all $j\in I \setminus \{0,1\}$. Hence, there is a bijection between the elements of the two geometries, and that bijection is clearly incidence preserving.
\end{proof}

\begin{lemma}\label{lem:residueNot01}
    Let $\Phi$ be a flag of $\PP(\Gamma)$ of cotype strictly containing $\{0,1\}$ and let $F$ be the flag of $\Gamma$ obtained by forgetting the partitions of $\Phi$. Then, $\PP(\Gamma)_{\Phi}$ is isomorphic to $\PP(\Gamma_F)(0,1)$ if $\Gamma_F[0,1]$ is not bipartite and $\PP(\Gamma)_{\Phi}$ is isomorphic to $\BP(\Gamma_F)(0,1)$ if $\Gamma_F[0,1]$ is bipartite.
\end{lemma}

\begin{proof}
    First of all, $\Gamma_F$ is a geometry that satisfies the conditions $(B_1)$ and $(B_2)$.
    Let $K = I \setminus (\{0,1\} \cup J)$ be the set of types not in $J$ and not in $\{0,1\}$. Our assumption implies that $K$ is not empty.

    If $\cap_{j\in J} P_{x_j}$ is equal to the vertex set of $\GG_F$, the graph of the $\{0,1\}$-truncation of $\Gamma_F$, then $\GG_F$ is not bipartite. In that case, we use Construction~\ref{constr:notBipartite} for building $\PP(\Gamma_F)(0,1)$. Else, $\GG_J$ is bipartite and we use Construction~\ref{BPRH} for building $\BP(\Gamma_F)(0,1)$.
    
    In both cases, the proof is similar to the proof of Lemma~\ref{lem:residue01}. For example, in the non bipartite case, both $\PP(\Gamma)_{\Phi}$ and $\PP(\Gamma_F)(0,1)$ have $\{(p,0) \mid p \in \cap_{j \in J} P_{x_j}\}$ as $0$-elements, $\{(q,1) \mid q \in \cap_{j \in J} \Bar{P}_{x_j}\}$ as $1$-elements and $\{(x_k,P_{x_k}) \mid x_k \in \Gamma_F, P_{x_k} \cap P_{x_j} \neq \emptyset, j\in J\}$ as $k$-elements for $k \in K$.

    It is easy to show that $\PP(\Gamma)_{\Phi}$ is isomorphic to the subgeometry of $\Gamma_F$ whose $0$-elements are in $\cap_{j \in J} P_{x_j}$ and whose $1$-elements are in $\cap_{j \in J} \Bar{P}_{x_j}$. The incidences are also the same in both geometries.
\end{proof}

\subsection{Residual connectedness}\hfill
\begin{prop}\label{prop:RC}
     Let $\Gamma$ be a residually connected geometry such that $\Gamma[0,1]$ is not bipartite. Then, $\PP(\Gamma)$ is residually connected.
\end{prop}

\begin{proof}
    Suppose first that $\Gamma$ is residually connected. Then, $\Gamma[0,1]$ must be connected, and thus, by Lemma~\ref{lem:PnbghGeom}, $\PP(\Gamma)[0,1]$ is also connected. Suppose $(x_j,P_{x_j})$ and $(x_k,P_{x_k})$ are two elements of $\PP(\Gamma)$ of type $j$ and $k$, respectively. Since $\PP(\Gamma)$ is a geometry, we can always find some $(p,0)$ incident to $(x_j,P_{x_j})$ and some $(p',0)$ incident to $(x_k,P_{x_k})$. We can then find a path from $(p,0)$ to $(p',0)$ in $\PP(\Gamma)[0,1]$. Concatenating these paths, we obtain a path from $(x_j,P_{x_j})$ to $(x_k,P_{x_k})$. Hence, $\PP(\Gamma)$ is connected. 
    
    Putting together, Lemma~\ref{lem:residue0}, Lemma~\ref{lem:residue01}, Lemma~\ref{lem:residue0J} and Lemma~\ref{lem:residueNot01}, we see that all the residues of corank at least two of $\PP(\Gamma)$ are connected, as they are either isomorphic to residues of residually connected geometries or isomorphic to a partitioned neighborhood geometry of a connected graph, which is connected by Lemma~\ref{lem:PnbghGeom}.

\end{proof}

\subsection{Automorphisms of $\PP(\Gamma)$}\hfill
\\

We know investigate the automorphism group of the partitioned geometry $\PP(\Gamma)$.
Let $\sigma$ be an automorphism of $\Gamma$. Then, we define $\PP(\sigma)$ to be the map sending $(p,i)$ to $(\sigma(p),i)$, $i =0,1$ and sending $(x,P)$ to $(\sigma(x),\sigma(P))$. Note that it makes sense to talk about $\sigma(P)$ as $P$ is a collection of $0$-elements of $\Gamma$. It is straightforward to check that $\PP(\sigma)$ is incidence preserving and bijective on elements of $\PP(\Gamma)$.

\begin{prop}\label{prop:autom}
    Let $G$ be a group of automorphisms of $\Gamma$ acting transitively on the chambers of $\Gamma$. Then $G$ also acts transitively by automorphisms on the chambers of $\PP(\Gamma)$.
\end{prop}
\begin{proof}
    Let $\Theta = \{(p,0),(q,1),(x_2,P_{x_2}),\cdots,(x_{n-1},P_{x_{n-1}}) \}$ be a chamber of $\PP(\Gamma)$. Then $C = \{p,e,x_2,\cdots,x_{n-1}\}$, where $e = \{p,q\}$, is a chamber of $C$.
    Let $\Theta' = \{(p',0),(q',1),(x_2',P_{x_2'}),\cdots,(x_{n-1}',P_{x_{n-1}'}) \}$ be any other chamber of $\PP(\Gamma)$ and let $C' = \{p',e',x_2',\cdots,x_{n-1}'\}$, wherer $e'= \{p',q'\}$. There is $\sigma \in G$ such that $\sigma(C) = C'$. Hence, we get that $\sigma(p) = p', \sigma(e) = e'$ and $\sigma(x_i) = x_i'$ for $i = 2,\cdots,n-1$. Therefore, $\sigma(q)=q'$. Also, $\sigma(P_{x_i}) = P_{x_i'}$ since both $\sigma(P_{x_i})$ and $P_{x_i'}$ must contain $p'$. Hence, we have that $\sigma(\Theta) = \Theta'$.
\end{proof}

The geometry $\PP(\Gamma)$ also has a duality (i.e: a proper correlation of order $2$), exchanging type $0$ and type $1$ elements.

\begin{prop}
    The map $\alpha$ defined by $\alpha(p,i) =(p, i+1 \pmod 2)$ and $\alpha(x,P) = (x,\Bar{P})$ is a correlation of $\PP(\Gamma)$.
\end{prop}

\begin{proof}
    The essential thing to note it that if $(p,i) *' (x,P)$ then $(p,i+1 \pmod 2) *'(x,\Bar{P})$.
\end{proof}

\subsection{Thinness}\hfill
\\

We quickly verify that being thin is preserved by the partioned geometry.
\begin{prop}\label{prop:thin}
    The geometry $\PP(\Gamma)$ is thin if and only if $\Gamma$ is thin.
\end{prop}

\begin{proof}
    In order to check that a geometry is thin, we need to check that all the residues of flags of corank one contains exactly two elements.

    Let $i \in I$. Lemma~\ref{lem:residue0J} tells us that the residues of flags $\Phi$ of cotype $i$ are all isomorphic to residues of some flag $F$ of cotype $i$ in $\Gamma$. Conversely, if $F$ is a flag of cotype $i$ in $\Gamma$, there always exists a flag $\Phi$ of cotype $i$ in $\PP(\Gamma)$ such that $\PP(\Gamma)_{\Phi}$ is isomorphic to $\Gamma_F$.
\end{proof}

\subsection{Conditions $(B_1)$ and $(B_2)$ in $\PP(\Gamma)$}\label{subsec:PPb1b2}\hfill

We now investigate whether the Construction~\ref{constr:notBipartite} can be used more than once or not. 

\begin{prop}\label{prop:b1b2}
    Suppose that $\Gamma$ satisfies conditions $(B_1)$ and $(B_2)$ for some pair $\{k,l\} \subset I \setminus \{0,1\}$. Then $\PP(\Gamma) = \PP(\Gamma)(0,1)$ also satisfies $(B_1)$ and $(B_2)$ for the pair $\{k,l\}$ if and only if, for all $k$-elements $x$ and $l$-elements $y$ with $x*y$, we have that either $\Gamma_x[0,1]$ is bipartite or both $\Gamma_x[0,1]$ and $\Gamma_y[0,1]$ are not bipartite, where all $\{0,1\}$-truncations as seen as graphs.
\end{prop}

\begin{proof}
    First, since $\Gamma$ satisfies $(B_1)$ and $(B_2)$ for the pair $\{k,l\}$, it means that the diagram $D$ of $\Gamma$ has a leaf $(k,l)$ with endpoint $k$. Therefore, for $x$ and $y$ as in the statement of the proposition, we have that $\Gamma_y[0,1]$ is a subgraph of $\Gamma_x[0,1]$. The condition of the proposition can thus be rephrased as either both $\Gamma_y[0,1]$ and $\Gamma_x[0,1]$ are bipartite or neither are. 

    Hence, suppose that either both $\Gamma_x[0,1]$ and $\Gamma_y[0,1]$ are bipartite or neither are.
    Let us show that $\PP(\Gamma)(0,1)$ satisfies $(B_1)$ for $\{k,l\}$. Let $(y,P_y)$ be an element of type $l$ of $\PP(\Gamma)(0,1)$.
    Since $\Gamma$ satisfies $(B_1)$ at $\{k,l\}$, there are exactly two elements $x_1$ and $x_2$ of type $k$ incident to $y$ and $y$ is the only element of type $l$ incident to both $x_1$ and $x_2$. Then, in both cases, there exist $(x_1,P_{x_1})$ and
    $(x_2,P_{x_2})$ incident to $(y,P_y)$. Suppose that there is a third element $(z,P_z)$, of type $k$ also, incident to $(y,P_y)$. This means that $z * y$ in $\Gamma$ and thus $z = x_1$ or $z = x_2$, else we would contradict $(B_1)$ for $\Gamma$. But the only way that $(y,P_y)$ is incident to both $(x_i,P_{x_i})$ and $(x_i,\bar{P}_{x_i})$ ($i\in\{1,2\}$) is if $\Gamma_y[0,1]$ is not bipartite but $\Gamma_{x_i}[0,1]$ is. Suppose  now that $(y',P_{y'})$ is an element of type $l$ in $\PP(\Gamma)$ that is also incident to both $(x_1,P_{x_1})$ and
    $(x_2,P_{x_2})$. Then, we get that $y'$ is incident to both $x_1$ and $x_2$ in $\Gamma$, so that $y'=y$. It also follows that $P_y = P_{y'}$. Indeed, since $\Gamma_y[0,1]$ is a subgraph of $\Gamma_{x_1}[0,1]$, the fact that $P_{y'} \cap P_{x_1} \neq \emptyset$ implies that $P_{y'} = P_{x_1} \cap \Gamma_y[0,1]$. Since the same holds for $P_y$, we can conclude that $P_y = P_{y'}$.

    We now show that $(B_2)$ holds for $\{k,l\}$ in $ \PP(\Gamma)(0,1)$. Let $(y,P_y)$ be of type $l$ and let $(x_1,P_{x_1})$ and $(x_2,P_{x_2})$ be the unique two elements of type $k$ incident to $(y,P_y)$. Since the diagram of $\PP(\Gamma)(0,1)$ has a leaf at $\{k,l\}$, every elements $(z,P_z)$ of type $i$ incident to $(y,P_y)$ is also automatically incident to both $(x_1,P_{x_1})$ and $(x_2,P_{x_2})$. Conversely, suppose that $(z,P_z)$ is incident to both $(x_1,P_{x_1})$ and $(x_2,P_{x_2})$. Then $z$ must be incident to both $x_1$ and $x_2$ and thus, by $(B_2)$ for $\Gamma$, also to $y$. But then $(z,P_z)$ is incident to $(y,P_y)$. Indeed, $P_z \cap P_y \neq \emptyset$ since, once again, $\Gamma_y[0,1]$ is a subgraph of $\Gamma_{x_1}[0,1]$ and $P_y \cap P_{x_1} \neq \emptyset$.

    Conversely, suppose instead that $\Gamma_y[0,1]$ is bipartite but $\Gamma_x[0,1]$ is not. Then, both $(y,P_y)$ and $(y,\Bar{P}_y)$ are both incident to both $(x_1,P_{x_1} = \Bar{P}_{x_1})$ and $(x_2,P_{x_2} = \Bar{P}_{x_2})$ so that there are double edges, failing $(B_1)$.
\end{proof}

\subsection{Main result}\hfill
\\

We summarize the conclusion of this section in one theorem.

\begin{thm}\label{thm:main1}
    Let $\Gamma$ be a residually connected geometry of rank $n \geq 3$ over the type set $I = \{0,1,\cdots,n-1\}$ that satisfies $(B_1)$ and $(B_2)$ and such that $\Gamma[0,1]$ is not a bipartite graph. If the diagram of $\Gamma$ is denoted as Figure~\ref{fig:leafDiagram}, there exists a geometry $\PP(\Gamma)(0,1)$ with diagram

    
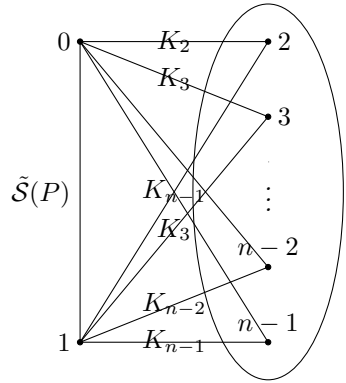
\begin{center}
         \begin{tikzpicture}[scale = 0.5]
    
         \filldraw[black] (2,-4) circle (2pt)  node[anchor=east]{$1$};
         \filldraw[black] (2,4) circle (2pt)  node[anchor=east]{$0$};
         \filldraw[black] (7,0.8) circle (0pt)  node[anchor=north]{$\vdots$};
         \filldraw[black] (7,2) circle (2pt)  node[anchor=west]{$3$};
         \filldraw[black] (7,-2) circle (2pt)  node[anchor=south]{$n-2$};
         \filldraw[black] (7,4) circle (2pt)  node[anchor=west]{$2$};
         \filldraw[black] (7,-4) circle (2pt)  node[anchor=south]{$n-1$};
         \draw (2,4) -- (2,-4)node [midway,left] (TextNode) {$\tilde{\mathcal{S}}(P)$};
         \draw (2,4) -- (7,2)node [midway] (TextNode) {$K_3$};
         \draw (2,-4) -- (7,2)node [midway] (TextNode) {$K_3$};
         \draw (2,4) -- (7,4)node [midway] (TextNode) {$K_2$};
         \draw (2,-4) -- (7,4)node [midway] (TextNode) {};
         \draw (2,4) -- (7,-2)node[midway] (TextNode) {};
         \draw (2,-4) -- (7,-2)node[midway] (TextNode) {$K_{n-2}$};
         \draw (2,4) -- (7,-4)node[midway] (TextNode) {$K_{n-1}$};
         \draw (2,-4) -- (7,-4)node[midway] (TextNode) {$K_{n-1}$};
         \draw (7,0) ellipse (2cm and 5cm);

    \end{tikzpicture}
    \end{center}

    where $\tilde{\mathcal{S}}(P)$ is the set $\{\tilde{\GG}(P) \mid \GG \in S, P \in \pi(\GG)\}$ and the diagram of $I \setminus \{0,1\}$ is the same as for $\Gamma$.
    Furthermore, the following properties hold:
    \begin{itemize}
    
        \item $\PP(\Gamma)(0,1)$ is residually connected.
        \item If $G$ is a group of automorphisms of $\Gamma$ acting flag-transitively on the chambers, then $G$ is also a group of automorphisms of $\PP(\Gamma)(0,1)$ acting flag-transitively on the chambers. Moreover, if the action on $\Gamma$ was faithful, then the action on $\PP(\Gamma)(0,1)$ is also faithful. 
        \item $\PP(\Gamma)(0,1)$ is thin if and only if $\Gamma$ is thin.
        \item Suppose there is some pair $\{k,l\} \subset I\setminus \{0,1\}$ for which conditions $(B_1)$ and $(B_2)$ hold. Then $\PP(\Gamma)(0,1)$ also satisfies $(B_1)$ and $(B_2)$ for the pair $\{k,l\}$ if and only if, for all $k$-elements $x$ and $l$-elements $y$, with $x*y$, we have that $\Gamma_x[0,1]$ is bipartite or both $\Gamma_x[0,1]$ and $\Gamma_y[0,1]$ are not bipartite, where all $\{0,1\}$-truncations are seen as graphs.
    \end{itemize}
\end{thm}

\begin{proof}
    This is a consequence of Theorem~\ref{PGammageom} and Propositions~\ref{prop:RC}, ~\ref{prop:autom}, ~\ref{prop:thin} and 
    ~\ref{prop:b1b2}.
\end{proof}

This result is in some sense a generalization of Theorem 4.1 of \cite{Leemans2000}, as the hypothesis on the connectdness of the neighborhood geometries is dropped.
Note that if $|\pi(G)| = 1$ for every $G \in \GG$, we have that $\Tilde{S}(P)$ is just the set $\Tilde{\GG}$ of neighborhood geometries of $\GG$, as defined in \cite{Leemans2000}. 
Also, in that case, it must be that all $G \in \GG$ are not bipartite, and hence contain at least one odd cycle. We thus precisely recover the Construction~\ref{NBPRH} of \cite{Leemans2000}.

\section{The halving geometry and non-degenerate leaf hypertopes}\label{sec:halvinggeom}

The goal of this section is to define a new geometry based on the partitioned geometry, which we will call the \emph{halving geometry}. We will also show that the halving geometry of a non-degenerate leaf regular hypertope $\Gamma$ is isomorphic to the geometry obtain as the coset geometry of the halving group of $\Gamma$.

Recall a similar result to Theorem~\ref{thm:main1}, from \cite{Leemans2000}, for the construction $\BP(\Gamma)(0,1)$ given in Construction~\ref{BPRH}.

\begin{thm}\label{thm:BP_fromDimi}
    Let $\Gamma$ be a geometry of rank $n \geq 3$ over the type set $I = \{0,1,\cdots,n-1\}$ that satisfies $(B_1)$ and $(B_2)$ and such that $\Gamma[0,1]$ is bipartite graph with bipartition $\{X_0,X_1\}$. If the diagram of $\Gamma$ is denoted as Figure~\ref{fig:leafDiagram}, there exists a geometry $\BP(\Gamma)(0,1)$ with diagram

    
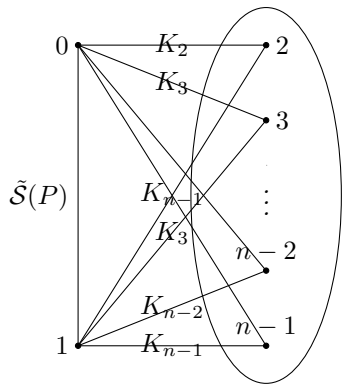
\begin{center}
         \begin{tikzpicture}[scale = 0.5]
    
         \filldraw[black] (2,-4) circle (2pt)  node[anchor=east]{$1$};
         \filldraw[black] (2,4) circle (2pt)  node[anchor=east]{$0$};
         \filldraw[black] (7,0.8) circle (0pt)  node[anchor=north]{$\vdots$};
         \filldraw[black] (7,2) circle (2pt)  node[anchor=west]{$3$};
         \filldraw[black] (7,-2) circle (2pt)  node[anchor=south]{$n-2$};
         \filldraw[black] (7,4) circle (2pt)  node[anchor=west]{$2$};
         \filldraw[black] (7,-4) circle (2pt)  node[anchor=south]{$n-1$};
         \draw (2,4) -- (2,-4)node [midway,left] (TextNode) {$\tilde{\mathcal{S}}(P)$};
         \draw (2,4) -- (7,2)node [midway] (TextNode) {$K_3$};
         \draw (2,-4) -- (7,2)node [midway] (TextNode) {$K_3$};
         \draw (2,4) -- (7,4)node [midway] (TextNode) {$K_2$};
         \draw (2,-4) -- (7,4)node [midway] (TextNode) {};
         \draw (2,4) -- (7,-2)node[midway] (TextNode) {};
         \draw (2,-4) -- (7,-2)node[midway] (TextNode) {$K_{n-2}$};
         \draw (2,4) -- (7,-4)node[midway] (TextNode) {$K_{n-1}$};
         \draw (2,-4) -- (7,-4)node[midway] (TextNode) {$K_{n-1}$};
         \draw (7,0) ellipse (2cm and 5cm);

    \end{tikzpicture}
    \end{center}

    where $\tilde{\mathcal{S}}(P)$ is the set $\{\tilde{\GG}(P) \mid \GG \in S, P \in \pi(\GG), P\cap X_0\neq \emptyset\}$ and the diagram of $I \setminus \{0,1\}$ is the same as for $\Gamma$.
    Furthermore, the following properties hold:
    \begin{itemize}
    
        \item $\BP(\Gamma)(0,1)$ is (residually) connected if and only if $\Gamma$ is (residually) connected.
        \item if $H$ is a group of automorphisms of $\BP(\Gamma)(0,1)$ acting transitively on the chambers and if there is an automorphism $\alpha$ of $\BP(\Gamma)(0,1)$ permuting the types 0 and 1 and fixing
        $I\backslash\{0,1\}$ pointwise, then $\langle H, \alpha \rangle$ acts as a chamber-transitive automorphism group of $\Gamma$. Conversely, if $G$ is a chamber-transitive automorphism group of $\Gamma$, then $G = H\rtimes C_2$ for some $H$ acting as a chamber transitive automorphism group of $\BP(\Gamma)(0,1)$.
        \item $\BP(\Gamma)(0,1)$ is thin if and only if $\Gamma$ is thin.
        \item $\BP(\Gamma)(0,1)$ satisfies $(B_1)$ and $(B_2)$ for the pair $\{k,l\}\subset I\backslash\{0,1\}$ if and only if $\Gamma$ does.
    \end{itemize}
\end{thm}

\begin{proof}
    This is a direct consequence of \cite[Theorem 6.2]{Leemans2000}, except for thinness. Thinness is straightforward, following a similar argument as in Proposition~\ref{prop:thin} and its residues as presented in the proof of \cite[Theorem 6.1]{Leemans2000}.
\end{proof}

\begin{definition}
 Let $\Gamma$ be a geometry of rank $n \geq 3$ over the type set $I = \{0,1,\cdots,n-1\}$ that satisfies $(B_1)$ and $(B_2)$.
 The \emph{halving geometry} $H(\Gamma)(0,1)$ is the geometry $\BP(\Gamma)(0,1)$ if $\Gamma[0,1]$ is bipartite, and $\PP(\Gamma)(0,1)$ otherwise.
\end{definition}

Notice that choosing the pair $\{0,1\}$ in the beginning of Section~\ref{sec:partitionedgeom} was simply for ease of notation. Remark that, for any pair $\{k,l\}\subset I$ that satisfies the obvious corresponding $(B_1)$ and $(B_2)$ conditions, we can build its halving geometry $H(\Gamma)(k,l)$.
Moreover, let $\Gamma$ have two pairs $\{0,1\}$ and $\{k,l\}$ that satisfy $(B_1)$ and $(B_2)$. We will denote $H(H(\Gamma)(0,1))(k,l)$ simply by $H(\Gamma)(0,1)(k,l)$.

\subsection{Non-degenerate leaf hypertopes}\hfill
\\

From now on, let $\Gamma$ be a regular hypertope of rank $n$ and let $G$ be its automorphism group.
We say that $\Gamma$ is a \textit{leaf hypertope} if its diagram $D(\Gamma)$ has at least one leaf (i.e: a vertex of degree $1$ as in Figure~\ref{fig:leafDiagram}).
For $\{i,j\} \subset I$ and $i\neq j$, we say that $\Gamma$ is a \textit{non-degenerate $(i,j)$-leaf hypertope} if $\Gamma$ satisfies $(B_1)$ and $(B_2)$ for the pair $\{i,j\}$.

Furthermore, we will suppose that $\Gamma$ is a non-degenerate leaf hypertope with a chosen leaf $(0,1)$, so that its diagram $D(\Gamma)$ is as shown in Figure \ref{fig:leafDiagram}, where $D(\Gamma)$ restricted to $\{2,3,\cdots,n-1\}$ is any graph.
As before, we will then say that the elements of type $0$ (respectively type $1$) are the vertices (respectively edges) of $\Gamma$.

Let $G$ be the automorphism group of $\Gamma$ and let $C$ be a base chamber of $\Gamma$. We remind the reader that $G = \langle \rho_0, \rho_1, \cdots, \rho_{n-1} \rangle$ where $\rho_i$ is defined to be the unique automorphism sending $C$ to its unique $i$-adjacent chamber $C^i$. The set $S = \{\rho_0, \rho_1, \cdots, \rho_{n-1}\}$ is then called the \textit{standard set of generators} with respect to $C$.
Consider the halving group $H(G) := \langle \tilde{\rho}_0, \rho_1, \rho_2, \ldots, \rho_{n-1}\rangle$, where $\tilde{\rho}_0 := \rho_0\rho_1\rho_0$, and let $\GG$ the $\{0,1\}$-truncation of $\Gamma$, seen, as usual, as a graph.

\begin{prop}\label{bipartite_cond}
    Let $\Gamma$ be a regular, non-degenerate $(0,1)$-leaf hypertope with automorphism group $G$. Let $C$ be a base chamber of $\Gamma$ and let $S = \{\rho_0, \rho_1, \cdots, \rho_{n-1}\}$ be the set of standard generators of $G$ with respect to $C$. Then, the following are equivalent
    \begin{enumerate}
        \item The graph $\GG$ is a bipartite graph.
        \item The halving group $H(G)$ is a proper subgroup of $G$ of index $2$.
        \item Let $G = \langle S \mid R \rangle$ be the standard presentation for $G$. Then every relator $r \in R$ contains an even number of $\rho_0$'s.
        \item The map $\varphi \colon S \to \mathbb{Z}_2$ sending $\rho_0$ to $1$ and $\rho_i$ to $0$, for $i \neq 0$, extends to a homomorphism from $G$ to $\mathbb{Z}_2$.
    \end{enumerate}
\end{prop}
\begin{proof}
    (c) $\iff$ (d): For the map $\varphi$ to extend to a homomorphism from $G$ to $\mathbb{Z}_2$, it is necessary and sufficient to check that for every relator $r \in R$, we have that $\varphi(r) = 0$. This is clearly the case if and only if every relator contains an even number of occurrences of $\rho_0$. 
    
    (d) $\implies$ (b): The halving group $H(G) = \langle \rho_0\rho_1\rho_0, \rho_1, \cdots, \rho_{n-1} \rangle$ is always either equal to $G$ or a proper subgroup of $G$. Clearly, it is equal to $G$ if and only if $\rho_0 \in H(G)$. Suppose that (d) holds. Then, clearly, $H(G) \subset \ker(\varphi)$, so that $H(G)$ must be an index $2$ subgroup of $G$.

    (a) $\iff$ (c): Every word $w = \prod_{j=1}^m \rho_{k_j}$, where $k_j \in \{0,1,\cdots,n-1\}$ for every $j =1,2,\cdots m$, in the alphabet $S$ corresponds to a path in $\GG$ in the following way. Let $C= \{v,e,F_2, \cdots, F_{n-1}\}$ be the base flag of $\Gamma$ and let $\rho_i$ act on $\CC(\Gamma)$ by sending each chamber $C$ to its $i-$adjacent chamber $C^i$. Then any word $w = \prod_{j=1}^m \rho_{k_j}$ corresponds to a path $C = C_0 - C_1 - \cdots - C_{m-1} - C_m = w(C)$ of chambers where $C_{j}$ is the $k_j$-adjacent chamber to $C_{j-1}$. This gives a sequence of $v_0,v_1, \cdots, v_k$ of vertices where $v_i$ is the $0$-element of $C_i$. We can reduce that sequence to a sequence $v'_0, v'_1, \cdots, v'_j$ where every $v'_i$ is distinct. Remark that each $v'_i$ is adjacent to $v'_{i+1}$ in $\GG$ and that $j$ is equal to the number of occurrences of $\rho_0$ in $w$, since $\rho_0$ is the only generator changing the vertex of the chambers.
    
    Now, if $w$ is a word that represents the identity if $G$, the path we obtain this way must be a cycle.
    Since there are no odd cycles in a bipartite graph, any such word must contain an even number of occurrences of $\rho_0$. Hence every relator $r \in R$ has an even number of occurrences of $\rho_0$'s. 
    
    Conversely, suppose all relators $r\in R$ contain an even number of $\rho_0$'s and, for a contradiction, suppose there is an odd cycle in $\GG$. This would imply that there is a word $w$ with an odd number of $\rho_0$'s that maps to the identity, a contradiction. Hence $\GG$ must be bipartite.

    (b) $\implies$ (c): The halving group $H(G)$ can be seen as the subgroup (not always proper) of $G$ containing all elements $g \in G$ that can be written as a word in $S$ using an even number of $\rho_0$'s. Suppose that $H(G)$ is a proper subgroup of index $2$ of $G$ and, for a contradiction, assume there is a $r\in R$ such that $r$ is a word with odd number of $\rho_0$'s.
    We have that $G$ can be partitioned into the cosets $G/H(G) = \{H(G), H(G)\rho_0\}$ and, hence, an element $g\in H(G)$ cannot be written at the same time with both an even and odd numbers of $\rho_0$'s. However, as $r$ is a word $w\rho_0w'=id$, where both $w,w'$ are words written with an even number of $\rho_0$'s, we have that $\rho_0 = w^{-1}w'^{-1}$, a contradiction. Hence all relators of $R$ are written with an even number of $\rho_0$'s.
\end{proof}

As we will need more flexible notations later, consider a regular hypertope $\Gamma$ of rank $n$ with a non-degenerate $(i,j)$-leaf, with $G=\Aut(\Gamma)=\langle \rho_0,\rho_1,\ldots,\rho_{n-1}\rangle$ and $\{i,j\}\subset \{0,1,\ldots,n-1\}$.
We can define the halving operation $\eta^{(i,j)}$ acting on $G$ as
$$\eta^{(i,j)} : \langle \rho_0,\ldots,\rho_i,\ldots,\rho_j,\ldots,\rho_{n-1}\rangle \rightarrow \langle \rho_0,\ldots,\rho_{i-1},\rho_i\rho_j\rho_i,\rho_{i+1},\ldots,\rho_j,\ldots,\rho_{n-1}\rangle $$ $$=: \langle \rho_0,\ldots,\rho_{i-1},\tilde{\rho_i},\rho_{i+1},\ldots,\rho_j,\ldots,\rho_{n-1}\rangle,$$
 When $(i,j)=(0,1)$, we have that $\eta^{(0,1)} = \eta$, the halving operation introduced by Montero and Weiss~\cite{MonteroWeiss}.
The resulting group after the operation $\eta^{(i,j)}$ will be denoted as $H(G)(i,j)$. For ease of notation, we will denote $H(G)(0,1) = H(G)$ when its implicit that the leaf is $(0,1)$.

We are now ready to prove the main theorem. It states that the coset geometry $(H(G),(H(G)_i)_{i \in I}))$ obtained from the halving group $H(G)$ together with its standard maximal parabolic subgroups is in fact the halving geometry $H(\Gamma)(0,1)$. This shows that it must be a regular hypertope.

 \begin{thm}\label{thm:main2}
    Let $\Gamma = (G,(G_i)_{i \in I})$ be a regular hypertope that satisfies $(B_1)$ and $(B_2)$ at $\{0,1\}$. Then $(H(G),(H(G)_i)_{i \in I}))$ is isomorphic to $H(\Gamma)(0,1)$.
\end{thm}

\begin{proof}
    Suppose that $\Gamma[0,1]$ is not bipartite. Then, by Proposition~\ref{bipartite_cond}, $G = H(G)$ and thus, by Theorem~\ref{thm:main1}, $H(G)$ acts simply flag-transitively on $H(\Gamma)$. Let $\Theta = \{(p,0),(q,1),(x_2,P_{x_2}),\cdots,(x_{n-1},P_{x_{n-1}}) \}$ be a chamber of $H(\Gamma)$ and let $\Theta^i, i \in I$, be the unique $i-$adjacent chamber to $\Theta$. We claim that $\Theta^0 =$ $\{(p',0),(q,1),$
    $(x_2,P_{x_2}),\cdots,(x_{n-1},P_{x_{n-1}}) \} = \rho_0 \rho_1 \rho_0 (\Theta)$ and that $\Theta^i = \rho_i(\Theta)$ for $i \neq 0$.
    Indeed, $\rho_0$ exchanges $p$ and $q$ while fixing $\{p,q\}$ and $x_i, i = 2,\cdots n-1$, and $\rho_1$ sends $\{p,q\}$ to $\{p,q'\}$, with $q' * x_i$ and $q'*p$, and thus $\rho_1$ fixes $p$ and sends $q$ to $q'$. Hence $\rho_0\rho_1\rho_0(p) = \rho_0\rho_1(q) = \rho_0(q') = p'$ where $p' * q$ and $p'*x_i, i = 2,\cdots n-1$. Also, $\rho_0\rho_1\rho_0(q) = \rho_0\rho_1(p) = \rho_0(p) = q$. This shows that $\rho_0\rho_1\rho_0(\Theta) = \Theta^0$.
    Similarly, $\rho_1(p) = p$ and $\rho_1(q) = q'$ so that $\rho_1(\Theta) = \Theta^1$.
    The theorem then holds by Proposition~\ref{prop:CosetGeometryfromAction}.

    Suppose instead that $\Gamma[0,1]$ is bipartite. Then, by Proposition~\ref{bipartite_cond}, we have that $|H(G)| = |G|/2$.
    We claim that $H(G)$ is exactly the subgroup of $G$ of elements acting on $H(\Gamma)$ without exchanging types. Indeed, $\rho_0$ exchanges the partition of $\Gamma[0,1]$, since it sends $p$ to $q$. Therefore, $\rho_0$ is a correlation of $H(\Gamma)$ that exchanges the elements of type $0$ and $1$. On the other hand, every $\rho_i$ with $i \neq 0$ act by automorphisms on $H(\Gamma)$. Therefore, any element $g \in H(G)$, having an even number of $\rho_0$ in any word representing it, is an automorphism of $H(\Gamma)$.
    By Theorem~\ref{thm:BP_fromDimi}, we know that the the automorphism group of $H(\Gamma)$ is precisely of index two in $G$, so that it must be $H(G)$. The end of the proof is exactly the same as in the non bipartite case.
\end{proof}

\subsection{On degeneracy conditions}\hfill
\\

In this section, we will explore algebraic counterparts to the conditions $(B_1)$ and $(B_2)$. We also investigate what happens if $\Gamma$ fails one of this conditions.

\begin{lemma}\label{algB1} Let $\Gamma$ be a $\{0,1\}$-leaf regular hypertope with $\Aut(\Gamma)=G$.
   $\Gamma$ satisfies $(B_1)$ if and only if $G_0\cap \rho_0 G_0\rho_0 = G_{0,1}$.
\end{lemma}
\begin{proof}

    Suppose first $\Gamma$ satisfies $(B_1)$. Hence, any $1$-element $e$ of $\Gamma$ can be identified with a pair of $0$-elements $\{p, q\}$ incident with $e$.
    Since $\Gamma$ is a regular hypertope, it is isomorphic to the coset geometry $\Gamma(G, (G_i)_{i \in I})$ where $I = \{ 0,1,\cdots,n-1)$, with $G = \Aut(\Gamma)$.
    It is obvious that $G_{0,1}\subset G_0$ and $G_{0,1}\subset \rho_0 G_0\rho_0$.
    Suppose, for a contradiction, that $G_{0,1} \subsetneqq G_0\cap \rho_0 G_0\rho_0$. Hence, there exists $x\in G_0\cap \rho_0 G_0\rho_0$ such that $x \notin G_{0,1}$. As $x\in G_0$, we must have $x\notin G_1$, meaning that $G_1x \neq G_1$. In $\Gamma(G, (G_i)_{i \in I})$, we have that $G_0\cap G_1 \neq \emptyset$ and $G_0\rho_0 \cap G_1 \neq \emptyset$. Moreover, $x\in G_0\cap G_1x$ and $x \in \rho_0G_0\rho_0\cap G_1x\subseteq \rho_0(G_0\rho_0 \cap G_1x)$. This means that $G_0\cap G_1x \neq \emptyset$ and $G_0\rho_0 \cap G_1x \neq \emptyset$. Hence, there are exactly two edges $G_1$ and $G_1x$ identified by the same pair of vertices, a contradiction. Hence $G_0\cap \rho_0 G_0\rho_0 = G_{0,1}$.

    Assume now that $G_0\cap \rho_0 G_0\rho_0 = G_{0,1}$. Let $C$ be the base chamber of $\Gamma$, and $C^0$ be the adjacent base chamber, changing only the type 0 element. The orbit of $G_0$ acting on the the base chamber gives a set of chambers fixing the type 0 element. Moreover, $\rho_0G_0\rho_0$ acting on the chamber $C^0$ fixes the type 0 element also.
    Hence, the subgroup $G_0\cap \rho_0 G_0\rho_0$ will have to fix both type 0-elements of chambers $C$ and $C^0$. As $G_0\cap \rho_0 G_0\rho_0 = G_{0,1}$, consequently fixing both these type 0 elements fixes also a type 1 element incident to both. Then, by regularity of $\Gamma$, we conclude that in the $\{0,1\}$-truncation of $\Gamma$, every pair of adjacent vertices define one and only one edge, satisfying $(B_1)$.
\end{proof}
We now take a look at $(B_2)$.

\begin{lemma}\label{algB2} Let $\Gamma$ be a $(0,1)$-leaf regular hypertope over $I$, with $\Aut(\Gamma)=G$, and such that $\Gamma)$ satisfies $(B_1)$.
If $G_0\rho_0\cap wG_0 = \emptyset$, for all $i\in I\backslash\{0,1\}$ and $w\in G_i$ such that $wG_0 \neq \rho_0G_0$, then $\Gamma$ satisfies $(B_2)$.
\end{lemma}
\begin{proof}
    Let $\Gamma$ be a $(0,1)$-leaf regular hypertope that satisfies $(B_1)$ and set $\Aut(\Gamma)=G$.
    The forward implication of condition $(B_2)$ is a direct consequence of the fact that $\Gamma$ is residually connected and that its rank two $\{0,i\}$ residues (with $i\in I\backslash\{0,1\}$) are digons (see \cite[Lemma 4.4]{Leemans2000}, for example).

    Consider a type $1$-element $e$ which, by $(B_1)$, can be identified by two $0$-elements $p$ and $q$. Let $x$ be an $i$-element, for $i\in I\backslash\{0,1\}$, such that $\{p,q\}\subseteq \sigma_0(x)$ and suppose that $e$ is not incident with $x$ (i.e. failing $(B_2)$).
    Let $C$ to be the base chamber of $\Gamma$, with $\{p,x\}\subset C$, and let $g(C)$ be the result of the action of $g\in G$ on the base chamber $C$. 
    
    As $q*x$, there is an element $w$ in $G_i$ (the stabilizer of $x$), such that $q\in w(C)$.
    Notice that if $G_0w = G_0\rho_0$, then $q$ is at distance one from $p$ in the $\{0,1\}$-truncation of $\Gamma_x$. This would mean that there is an edge $e'=\{p,q\}$, different from $e$, incident to $x$, which contradicts $(B_1)$. 
    
    Therefore, consider $w\in G_i$ such that $G_0w \neq G_0\rho_0$. Consider the set of all chambers where we fix both $q$ and $x$ as its $0$ and $i$-elements. This set can be obtained by acting on the base chamber with elements from $G_0w$. In other words, we look at the set $\CC_{G_0w} =\{m(C)\ |\ m\in G_0w\}$. 
    We know that there exists an $1$-element $e=\{p,q\}$, hence $q$ must also be adjacent to $p$ in the $\{0,1\}$-truncation of $\Gamma$, seen as a graph. 
    Consider then the set of all chambers in which we fix the $0$-element as $p$, the set $\CC_{G_0} =\{k(C)\ |\ k\in G_0\}$. If we change the $0$-element to all the chambers of $\CC_{G_0}$, we get a new set of chambers $\CC_{\rho_0G_0} =\{k(C)\ |\ k\in \rho_0G_0\}$.
     Since $p$ and $q$ are adjacent in the $\{0,1\}$-truncation of $\Gamma$, there is a common chamber between $\CC_{G_0w}$ and $\CC_{\rho_0G_0}$, i.e. the intersection $\CC_{G_0w}\cap \CC_{\rho_0G_0}$ is non-empty. This is equivalent to saying that $G_0\rho_0\cap wG_0 \neq \emptyset$.
\end{proof}

One important consequence of imposing condition $(B_1)$ is to guarantee that the halving group is a C-group as, otherwise, $H(G)$ will never be a C-group.

\begin{prop}
    Let $\Gamma$ be a regular $(0,1)$-leaf hypertope that does not satisfy condition $(B_1)$ for the chosen leaf. Then $H(G)(0,1)$ is not a C-group, where $G=\Aut(\Gamma)$.
\end{prop}
\begin{proof}
    Suppose $\Gamma$ is a regular leaf-hypertope that does not satisfies condition $(B_1)$ for a $(0,1)$-leaf.
Hence, by Lemma~\ref{algB1}, $G_0\cap \rho_0 G_0 \rho_0 \neq G_{0,1}$. As $G_{0,1}\subset G_0$ and $G_{0,1}\subset \rho_0 G_0\rho_0$ then there exists $x\in G_0\cap \rho_0 G_0 \rho_0$ such that $x\notin G_{0,1}$.
Consider then the halving group $H(G)(0,1) = H(G) = \langle \rho_0\rho_1\rho_0, \rho_1, \rho_2, \ldots \rho_{n-1}\rangle$, where $G = \Aut(\Gamma)$.
Hence we have the maximal parabolic subgroups $H(G)_0 = \langle \rho_1, \rho_2, \ldots\rho_{n-1}\rangle$ and $H(G)_1 = \langle \rho_0\rho_1\rho_0, \rho_2, \ldots \rho_{n-1}\rangle$.

As $H(G)_0 = G_0$, $H(G)_1= \rho_0 G_0 \rho_0$ and $H(G)_{0,1}=G_{0,1}$, we have $H(G)_0\cap H(G)_1 \neq H(G)_{0,1}$, failing the intersection property. Hence, $H(G)$ is not a C-group.
\end{proof}

Conditions $(B_1)$ and $(B_2)$ are weaker conditions to impose on a regular polytope $\Gamma$ than requiring its poset to be a lattice. Example~\ref{notlatticehyper} shows that it is possible for a regular polytope to have a face poset which is not a lattice but to satisfy conditions $(B_1)$ and $(B_2)$.

\begin{example}\label{notlatticehyper}
    Let $\Gamma$ be the regular polytope with Schl\"{a}fli type $\{3,3,2\}$. The presence of a $2$ in the Schl\"{a}fli type immediately guarantees that the face poset of the $\Gamma$ is not a lattice. Nonetheless, it is easily verifiable that $\Gamma$ satisfies conditions $(B_1)$ and $(B_2)$ for its $(0,1)$-leaf. Hence $H(\Gamma)$ gives a regular hypertope, even though its face poset is not a lattice.
\end{example}

Finally, there exists a case where condition $(B_2)$ is not satisfied and yet the coset geometry of the halving group gives a regular hypertope (see Example~\ref{example:hemicube}).

\begin{example}\label{example:hemicube}
    Let $\Gamma$ be the hemicube and $G=\Aut(\Gamma)$. Its $(2,1)$-leaf fails condition $(B_1)$ but the $(0,1)$-leaf satisfies $(B_1)$. Moreover, it can be easily seen that there exists two vertices $p,q$ of $\Gamma$ such that both are incident to a face $x$ but the edge $e=\{p,q\}$ is not incident to $x$, failing condition $(B_2)$. Nevertheless, the coset geometry of the halving group $H(G)(0,1)$, $\Gamma(H(G)(0,1),(H(G)(0,1)_i)_{i\in I})$, gives a regular polytope still (the tetrahedron).
\end{example}

These examples illustrate interesting implications of these properties. As we proved, $(B_1)$ is essential to guarantee that the halving group is still a C-group.
Condition $(B_2)$ is essential in our proof of Theorem~\ref{PGammageom}, showing that the halving geometry is in fact a geometry.
However Example~\ref{example:hemicube} shows that condition $(B_2)$ is not strictly necessary. In fact, the hemicube also fails the lattice condition required in the construction of Montero and Weiss~\cite{MonteroWeiss}. Nevertheless, this example seems to be a sporadic case where condition $(B_2)$ is not met but the halving geometry is still a regular hypertope.

\section{Regular Hypertopes from Cubic toroids}\label{sec:cubictoroid}

The cubic toroids have been used frequently as a starting point for building families of hypertopes \cite{piedade2023infinite,ens2018rank,MonteroWeiss}.
Here, we will use this family of regular polytopes as a base to apply our halving constructions. In Section~\ref{subsec:cubtor} we introduce the cubic toroids, discuss the non-degeneracy of their leafs and build the halving geometry $H(\Gamma)(0,1)$.
In Section~\ref{subsec:halvgroup} we give the group presentations for the automorphism groups of $H(\Gamma)(0,1)$ and $H(\Gamma)(0,1)(n,n-1)$. Finally, in Section~\ref{subsec:diagrams} we successively apply the halving geometry to non-degenerate leafs of the cubic toroids and give all the possible diagrams obtained by these operations.

\subsection{$(n+1)$-Cubic toroids}\label{subsec:cubtor}\hfill
\\

Consider the affine Coxeter group $[4,3^{n-2},4]$, where $q^{k}$, for any interger $q$, represents a string $q,\ldots,q$ of length $k$. This is the automorphism group of the tessellation $\mathcal{T} = \{4,3^{n-2},4\}$ of $\mathbb{E}^n$ by hypercubes, with vertex set $\mathbb{Z}^n$. Let $T$ be the translation subgroup of $[4,3^{n-2},4]$. Concretely, $T = \langle e_1, e_2, \ldots, e_n \rangle$, where $e_i$ is the $i$'th standard basis vector of $\mathbb{E}^n$ for $i\in\{1,\ldots,n\}$. 

Let $\Lambda_{\textbf{s}}$ be the subgroup of $T$ generated by a vector $\textbf{s}$ and all its conjugates under $[4,3^{n-2},4]$. We will define the $(n+1)$-regular polytope $\{4,3^{n-2},4\}_{\textbf{s}}$ as the one whose automorphism group is the string C-group $[4,3^{n-2},4]_{\textbf{s}} := [4,3^{n-2},4]/\Lambda_{\textbf{s}}$, where the vector $\textbf{s}\in\{(s,0^{n-1}), (s,s,0^{n-2}), (s^n)\}$.

\begin{thm}\label{thm:cubicPoly}\cite[Theorem 6D4]{ARP}
Let $n\geq 3$ and $s\geq 2$, with ${\bf s}=(s^k,0^{n-k})$ and $k\in\{1,2,n\}$. There is a self-dual regular polytope $\{4,3^{n-2},4\}_{\bf{s}}$ whose automorphism group is the Coxeter group $[4,3^{n-2},4] = \langle \rho_0,\ldots,\rho_n\rangle$ factored by the single extra relation $$(\rho_0\rho_1\ldots\rho_{n-1}\rho_n\rho_{n-1}\rho_{n-2}\ldots\rho_{k+1}\rho_k)^{ks} = id.$$ 
\end{thm}

In other words, the polytopes $\{4,3^{n-2},4\}_{\textbf{s}}$ are obtained by quotienting $\mathcal{T}$ by $\Lambda_{\textbf{s}}$. We make explicit here what $\Lambda_{\textbf{s}}$ is for each value of $\textbf{s}$.
For $\textbf{s}=(s,0^{n-1})$, the group $\Lambda_{\textbf{s}}$ is generated by $\{s e_1,s e_2, \ldots,s e_n\}$. For $\textbf{s}=(s,s,0^{n-2})$, the group $\Lambda_{\textbf{s}}$ is generated by $\{s (e_1+e_2), s(e_2-e_1), s(e_3-e_2), \ldots, s(e_n - e_{n-1})\}$, i.e. $$\{ (s,s,0^{n-2}), (-s,s,0^{n-2}), (0,-s,s,0^{n-3}),\ldots,(0^{n-2},-s,s)\}.$$
Finally, for $\textbf{s}=(s^n)$, the group $\Lambda_{\textbf{s}}$ is generated by $\{2se_1,2se_2,\ldots,2se_{n-1},s(e_1+\ldots+e_n)\}$, i.e. $$\{ (2s,0^{n-1}), (0,2s,0^{n-2}),\ldots,(0^{n-2},2s,0),(s^n)\}.$$

We first check for which values of $s$ and $k$ the $\{0,1\}$-truncation of $\{4,3^{n-2},4\}_{\bf{s}}$ is bipartite.
\begin{lemma}
    The $\{0,1\}$-truncation of $\{4,3^{n-2},4\}_{\bf{s}}$ is non-bipartite if and only if both $s$ and $k$ are odd. Otherwise, the $\{0,1\}$-truncation of $\{4,3^{n-2},4\}_{\bf{s}}$ is bipartite.
\end{lemma}
\begin{proof}
    By Proposition~\ref{bipartite_cond}, we just need check when all relations have an even number of $\rho_0$'s. This is clearly true if and only if either $k$ or $s$ are even.
\end{proof}
We now verify that the leafs of the polyopes $\{4,3^{n-2},4\}_{\bf{s}}$ are non degenerate.
\begin{lemma}\label{lem:NondegenerateLeaf}
Let $n\geq 3$ and $s\geq 2$, with ${\bf s}=(s^k,0^{n-k})$ and $k\in\{1,2,n\}$. The regular polytope $\{4,3^{n-2},4\}_{\bf{s}}$ has a degenerate leaf if and only if $k=1$ and $s=2$. 
\end{lemma}
\begin{proof}
    To satisfy condition $(B_1)$, we need to show that there is at most one edge incident to two given vertices. Clearly, this is true in the tessellation $\mathcal{T}$. Therefore, this will also be true in the polytope $\{4,3^{n-2},4\}_{\textbf{s}}$, as long as no double edge is created when taking the quotient of $\mathcal{T}$ by the action of $\Lambda_\textbf{s}$. A double edge is created only if we can find two vertices $v$ and $w$ of $\mathcal{T}$ such that they are at distance $2$ in the $\{0,1\}$-truncation of $\mathcal{T}$ and they get identified by $\Lambda_\textbf{s}$. This can only happen if the vector $\textbf{s}$ is equal to $(2,0^{n-1})$ or to $(1,1,0^{n-2})$. Since Theorem~\ref{thm:cubicPoly} requires $s \geq 2$, only the case of $\textbf{s} = (2,0^{n-1})$ can happen. Hence $(B_1)$ is not satisfied when $\textbf{s}=(2,0^{n-1})$.

    Assume now that condition $(B_1)$ is satisfied. For condition $(B_2)$, the forward implication is true since $\{4,3^{n-2},4\}_{\bf{s}}$ is residually connected and the rank two residues of type $\{0,i\}$, with $i\in I\backslash\{0,1\}$ are generalized digons (see \cite[Lemma 4.4]{Leemans2000}). For the reverse implication, we will prove by induction on the types $i\in\{2,\ldots,n\}$ that if $\sigma_0(e)\subseteq \sigma_0(x)$ then $e*x$.
    We remind the reader that the infinite tesselation $\{4,3^{n-2},4\}$ is vertex-describable since its automorphism group is the universal Coxeter group $[4,3^{n-2},4]$~\cite[Theorem 3D7(d)]{ARP}.
    Consider a 2-element $x$ and an edge $e$ such that $e=\{p,q\}$. Let $\{p,q\}\subseteq \sigma_0(x)$ and suppose, for a contradiction, that $e$ is not incident to $x$. Let $p$ be the vertex $(0^{n})$ and $x$ be the face incident with the vertices $(0^n)$, $(1,0^{n-1})$, $(0,1,0^{n-2})$ and $(1,1,0^{n-2})$. Due to $(B_1)$, $q\notin\{(1,0^{n-1}),(0,1,0^{n-2})\}$ as this would define a new edge $\{p,q\}$ different from $e$.
    Hence, $q=(1,1,0^{n-2})$. As $\{p,q\}$ identifies the edge $e$, $p$ will be identified with either vertices $(1,1,0^{j},1,0^{n-j-3})$, for $0 \leq j\leq n-3$, $(2,1,0^{n-2})$ or $(1,2,0^{n-2})$. If $p=(1,1,0^{j},1,0^{n-j-3})$, then we have that $s=1$, a contradiction.
    If $p$ is either $(2,1,0^{n-2})$ or $(1,2,0^{n-2})$ then $\{4,3^{n-2},4\}_{\textbf{s}}$ is not a regular polytope, a contradiction. Hence $e*x$.
    Let now $i\in\{3,\ldots,n\}$ and assume that for an element of type $i-1$ the induction hypothesis holds. Consider an element $x$ of type $i$ and an edge $e$ such that $e=\{p,q\}$. Let $\{p,q\}\subseteq \sigma_0(x)$ and suppose, for a contradiction, that $e$ is not incident to $x$. Consider that $p=(0^{n})$. By $(B_1)$ and the induction hypothesis, the only vertex possible for $q$ is the antipode of $p$ on the $i$-cube, which without lost of generality can be thought as $(1^i,0^{n-i})$. As before, since $\{p,q\}$ identifies the edge $e$, $p$ will be identified with either vertices that imply that $s=1$ or that $\{4,3^{n-2},4\}_{\textbf{s}}$ is not a regular polytope, a contradiction. Hence $e*x$.
\end{proof}

Consider the non-degenerate polytopes $\{4,3^{n-2},4\}_{\bf{s}}$. 
Before stating the general case, we make explicit the two simplest cases, which illustrate the general behaviour.

\begin{example}\label{example:(3,0,0)}
    Let ${\bf s}=(3,0,0)$. Then the polytope $\Gamma = \{4,3,4\}_{(3,0,0)}$ can be obtained by identifying the sides of a $3 \times 3 \times 3$ cube, as depicted in Figure \ref{fig:(3,0,0)}. Since the $\{0,1\}$-truncation of $\Gamma$ is not a bipartite graph, we cannot use Construction~\ref{BPRH} of \cite{Leemans2000}. We instead use our new Construction~\ref{constr:notBipartite}. So we double the vertices, and since the $\{0,1\}$-truncations of cubes and squares are bipartite, we also double the squares and cubes. A visual representation of a ``cube" of $\PP(\Gamma)(0,1)$ is shown in Figure \ref{fig:(3,0,0)}. Note that if we do not double the cubes and squares, the geometry would not be residually connected.
\end{example}

\begin{figure}\centering
    \begin{tikzpicture}
        \draw (0,0) -- (3,0);
        \draw (3,3) -- (3,0);
        \draw (0,3) -- (3,3);
        \draw (0,0) -- (0,3);

        \draw (0,3) -- (2,4);
        \draw (0+3,3) -- (2+3,4);
        \draw (0+3,3-3) -- (2+3,4-3);

        \draw (2+3,4-3) -- (2+3,4);
        \draw (2,4) -- (2+3,4);
        \foreach \i in {1,...,2}
        {
            \draw (\i,0) -- (\i,3);
            \draw (0,\i) -- (3,\i);
            \draw (\i,3) -- (\i+2,4);
            \draw (3,\i) -- (3+2,\i+1);
            \draw (3+ 2*\i/3,\i/3+3) -- (2*\i/3,\i/3+3);
            \draw (3+ 2*\i/3,\i/3) -- (3+2*\i/3,\i/3+3);
        }

        \draw[dotted] (2,2) --(2+2/3,2+1/3) -- (3+2/3,2+1/3);
        \draw[dotted] (2+2/3,2+1/3) --(2+2/3,3+1/3);
        \filldraw [olive] (3,3) circle (2pt);
        \filldraw [olive] (2,2) circle (2pt);
        \filldraw [olive] (3+2/3,2+1/3) circle (2pt);
        \filldraw [olive] (2+2/3,3+1/3) circle (2pt);
        \filldraw [teal] (2,3) circle (2pt);
        \filldraw [teal] (3,2) circle (2pt);
        \filldraw [teal] (3+2/3,3+1/3) circle (2pt);
        \filldraw [teal] (2+2/3,2+1/3) circle (2pt);

        \draw[dotted,teal] (2,2) -- (2+2/3,3+1/3) -- (3+2/3,2+1/3) --(3,3)--(2,2) ;
    \end{tikzpicture}
    \caption{The polytope $\Gamma = \{4,3,4\}_{(3,0,0)}$ is obtained by identifying the sides of the $3 \times 3 \times 3$ cube depicted in this figure. The $\{0,1\}$-truncation of $\Gamma$ is not bipartite, but the $\{0,1\}$-truncation of the residue of any cube or square is, as showed for the top-right cube.}
    \label{fig:(3,0,0)}   
\end{figure}
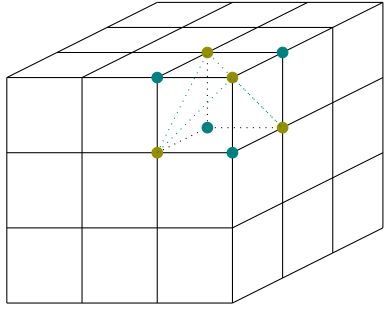

\begin{example}\label{example:(4,0,0)}
    Let $\textbf{s} = (4,0,0)$. In this case, the polytope $\Gamma = \{4,3,4\}_{(4,0,0)}$ can be obtained by identifying the sides of a $4 \times 4 \times 4$ cube. The $\{0,1\}$-truncation is bipartite, so we can use Construction~\ref{BPRH}. The hypertope $\BP(\Gamma)(0,1)$ has vertices of each color as elements of type $0$ and $1$, respectively, and faces and squares as elements of type $2$ and $3$, respectively (see Figure~\ref{fig:(4,0,0)}).
\end{example}

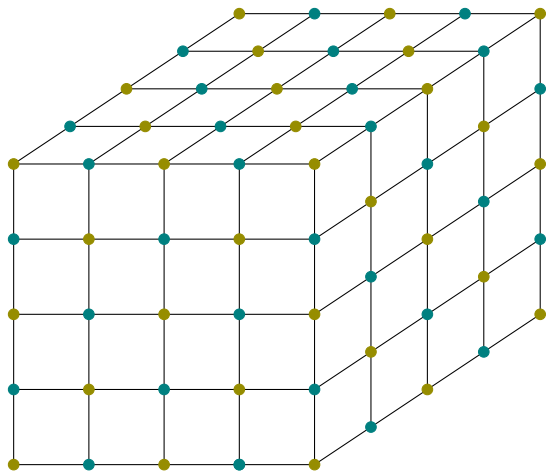
\begin{figure}\centering
    \begin{tikzpicture}
        \draw (0,0) -- (4,0);
        \draw (4,4) -- (4,0);
        \draw (0,4) -- (4,4);
        \draw (0,0) -- (0,4);

        \draw (0,4) -- (3,6);
        \draw (0+4,4) -- (3+4,6);
        \draw (0+4,0) -- (3+4,6-4);

        \draw (3+4,6-4) -- (3+4,6);
        \draw (3,6) -- (3+4,6);
        
        \foreach \i in {1,...,3}
        {
            \draw (\i,0) -- (\i,4);
            \draw (0,\i) -- (4,\i);
            \draw (\i,4) -- (\i+3,6);
            \draw (4,\i) -- (4+3,\i+2);
            
            \draw (4+ 3*\i/4,2*\i/4+4) -- (3*\i/4,2*\i/4+4);
            \draw (4+ 3*\i/4,2*\i/4) -- (4+3*\i/4,2*\i/4+4);
        }
        \foreach \i in {0,2,4}
        {
        \filldraw[olive] (0+3*\i/4,4+2*\i/4) circle (2pt);
        \filldraw[olive] (2+3*\i/4,4+2*\i/4) circle (2pt);
        \filldraw[olive] (4+3*\i/4,4+2*\i/4) circle (2pt);
        \filldraw[olive] (4+3*\i/4,2+2*\i/4) circle (2pt);
        \filldraw[olive] (4+3*\i/4,0+2*\i/4) circle (2pt);
        }
        \foreach \i in {1,3}
        {
        \filldraw[teal] (0+3*\i/4,4+2*\i/4) circle (2pt);
        \filldraw[teal] (2+3*\i/4,4+2*\i/4) circle (2pt);
        \filldraw[teal] (4+3*\i/4,4+2*\i/4) circle (2pt);
        \filldraw[teal] (4+3*\i/4,2+2*\i/4) circle (2pt);
        \filldraw[teal] (4+3*\i/4,0+2*\i/4) circle (2pt);
        }
        \foreach \i in {0,2,4}
        {
        \filldraw[teal] (1+3*\i/4,4+2*\i/4) circle (2pt);
        \filldraw[teal] (3+3*\i/4,4+2*\i/4) circle (2pt);
        \filldraw[teal] (4+3*\i/4,3+2*\i/4) circle (2pt);
        \filldraw[teal] (4+3*\i/4,1+2*\i/4) circle (2pt);
        }
        \foreach \i in {1,3}
        {
        \filldraw[olive] (1+3*\i/4,4+2*\i/4) circle (2pt);
        \filldraw[olive] (3+3*\i/4,4+2*\i/4) circle (2pt);
        \filldraw[olive] (4+3*\i/4,3+2*\i/4) circle (2pt);
        \filldraw[olive] (4+3*\i/4,1+2*\i/4) circle (2pt);
        }
        \foreach \x in {0,1,2,3}
        {
                \filldraw[olive](\x,\x)circle (2pt);        
        }
        \filldraw[olive](2,0)circle (2pt);    
        \filldraw[olive](0,2)circle (2pt);
        \filldraw[olive](1,3)circle (2pt);
        \filldraw[olive](3,1)circle (2pt);  
        \filldraw[teal](1,0)circle (2pt);
        \filldraw[teal](3,0)circle (2pt);
        \filldraw[teal](0,1)circle (2pt);
        \filldraw[teal](3,0)circle (2pt);
        \filldraw[teal](1,2)circle (2pt);
        \filldraw[teal](2,1)circle (2pt);
        \filldraw[teal](3,2)circle (2pt);
        \filldraw[teal](2,3)circle (2pt);
        \filldraw[teal](0,3)circle (2pt);
    \end{tikzpicture}
    \caption{The polytope $\Gamma = \{4,3,4\}_{(4,0,0)}$ is obtained by identifying the sides of the $4 \times 4 \times 4$ cube shown here. The $\{0,1\}-$truncation of $\Gamma$ is bipartite, hence $2$-colorable.}
    \label{fig:(4,0,0)}   
\end{figure}

We now prove the general case.

\begin{prop}\label{Bnhypertope}
    Let $\Gamma$ be the regular polytope $\{4,3^{n-2},4\}_{\bf{s}}$ with ${\bf s} \neq (2,0^{n-1})$. Then $H(\Gamma)(0,1)$ is a regular hypertope. Moreover, $H(\Gamma)(n,n-1)$ is dual of $H(\Gamma)(0,1)$. 
\end{prop}

\begin{proof}
    Lemma~\ref{lem:NondegenerateLeaf} shows that the $(0,1)$-leaf of $\Gamma$ is non-degenerate. Hence, $\Gamma$ satisfies all conditions of Theorem~\ref{thm:main1}. Since $\Gamma$ is self-dual, this means that the $(n,n-1)$-leaf of $\Gamma$ is also non-degenerate, implying that $H(\Gamma)(0,1)$ is the dual of $H(\Gamma)(n,n-1)$.
\end{proof}

We wish to apply our halving operation again on $H(\Gamma)(0,1)$. For that, we need to verify that the new leafs are still non-denegerate.

\begin{prop}\label{b1b2forhalvedcubic}
    Let $\Gamma$ be the regular polytope $\{4,3^{n-2},4\}_{\bf{s}}$ with ${\bf s} \notin \{(2,0^{n-1}),(2,2,0^{n-2})\}$. Then $H(\Gamma)(0,1)$ has exactly $3$ leafs and they are all non-degenerate.
\end{prop}

\begin{proof}
    The $(n,n-1)$-leaf is non-degenerate as a direct consequence of Proposition~\ref{prop:b1b2}.
    The other two leafs are $(0,2)$ and $(1,2)$. Consider then the leaf $(0,2)$. Regardless of whether we used Construction~\ref{BPRH} or \ref{constr:notBipartite} to go from $\{4,3^{n-2},4\}_\textbf{s}$ to $H(\Gamma)(0,1)$, two elements of type $0$ of $\Gamma$ are adjacent in the $\{0,2\}$-truncation of $\Gamma$ if and only if the corresponding vertices of $\{4,3^{n-2},4\}_\textbf{s}$ are at distance $2$ in the $\{0,1\}$-truncation of $\{4,3^{n-2},4\}_\textbf{s}$, seen as a graph. Therefore, the same argument used in Lemma~\ref{lem:NondegenerateLeaf}, shows that two elements of type $0$ of $H(\Gamma)(0,1)$ cannot possibly be incident to two distinct faces of $\Gamma$, unless if $\textbf{s}=(2,2,0^{n-2})$. For the leaf $(1,2)$, we arrive to the same conclusion. Hence when ${\bf s} \notin \{(2,0^{n-1}),(2,2,0^{n-2})\}$, $H(\Gamma)(0,1)$ satisfies $(B_1)$ for the leafs $(0,2)$ and $(1,2)$.
    Suppose then that $H(\Gamma)(0,1)$ satisfies $(B_1)$ for the leaf $(0,2)$. Following the proof of Lemma~\ref{lem:NondegenerateLeaf}, we can prove that $H(\Gamma)(0,1)$ satisfies $(B_2)$ for $i\geq 3$. For $i=1$, using again geometrical arguments on the infinite tessellation, it can be easily checked that $(B_2)$ fails if and only if $\textbf{s}=(2,2,0^{n-2})$. The same result can be obtained starting from the leaf $(1,2)$.
\end{proof}

\subsection{Halving groups presentations}\label{subsec:halvgroup}\hfill
\\

Here, we determine natural group relations of the halving group for the halving geometries $H(\Gamma)(0,1)$ and $H(\Gamma)(0,1)(n,n-1)$ coming from the cubic toroids. To do so, we will use Tietze transformations and the Reidemeister-Schreier subgroup presentation (see \cite{magnus2004combinatorial}, for example). When $\Gamma[0,1]$ is not bipartite, the automorphism group does not change. The goal is then to rewrite the relation with $\Tilde{\rho}_0$ instead of $\rho_0$. 

\begin{prop}\label{grouppres_odd}
    Let both $k$ and $s$ be odd and $\tilde{\rho}_0 = \rho_0\rho_1\rho_0$. Then, $\Gamma[0,1]$ is non-bipartite and the automorphism group of $H(\Gamma)(0,1)$ will be a quotient of the Coxeter group $[\nfrac{3}{3},3^{n-3},4] = \langle\tilde{\rho}_0,\rho_1,\ldots,\rho_n\rangle$ by the following extra relation
    $$ (\tilde{\rho}_0\rho_2\ldots\rho_{n-1}\rho_n\rho_{n-1}\ldots\rho_2\rho_1)^{2s} = id $$
    for $k=1$, and 
    $$ (\tilde{\rho}_0\rho_2\ldots\rho_{n-1}\rho_n\rho_1\rho_2\ldots\rho_{n-1}\rho_n)^{ns}=id$$
    for $k=n$.    
\end{prop}
\begin{proof}
    When $\Gamma$ is non-bipartite, we have that $\Aut(\Gamma)\cong \Aut(H(\Gamma))$. Let $S = \{\rho_0,\ldots,\rho_{n}\}$ be the set of generators of $\Aut(\Gamma)$ and $R$ be the set of relations of $\Aut(\Gamma)$. Hence $R$ is generated by the relations of the Coxeter group $[4,3^{n-2},4]$ plus the extra relation $(\rho_0\rho_1\ldots\rho_{n-1}\rho_n\rho_{n-1}\rho_{n-2}\ldots\rho_{k+1}\rho_k)^{ks} = id$ (see Theorem~\ref{thm:cubicPoly}).
    We will divide the proof into two cases. The first case is when $k=1$ and $s$ odd, and the second case is when $k=n$ and $s$ and $n$ are odd.

    \textbf{Case $k=1$ and $s$ is odd:}
    Let $u = \rho_0\rho_1\ldots\ldots\rho_2\rho_1$.
    Using a Tietze transformation, we can add the generators $\tilde{\rho}_0 = \rho_0\rho_1\rho_0$ and $U = (\rho_0\rho_1\ldots\rho_{n-1}\rho_n\rho_{n-1}\ldots\rho_1)^2$ to $S$. Notice that $U^{\frac{s-1}{2}} = u^{-1}$ is a word written with even number of $\rho_0$.
    As $s$ is odd, we can rewrite the extra relation in $R$ as $U^s=id$.
    
    We can rewrite $\rho_0 = \rho_1\rho_2\ldots\rho_{n-1}\rho_n\rho_{n-1}\ldots\rho_2\rho_1 U^{\frac{s-1}{2}}$, which implies that the relation $(\rho_0\rho_i)^2 = id$ can be rewritten as $(U^{-1}\rho_i)^2=id$, for $i\geq 2$. However, as $U$ commutes with $\rho_i$, for $i\geq 2$, these relations can be then removed from $R$ as they can be derived from the others. 
    
    For $i\geq 3$, conjugating the relations $(\rho_1\rho_i)^2=id$ and $(\rho_1\rho_2)^3$ by $\rho_0$ gives us relations $(\tilde{\rho}_0\rho_i)^2=id$ and $(\tilde{\rho}_0\rho_2)^3=id$. The relation $(\rho_0\rho_1)^4=id$ can be rewritten as $(\tilde{\rho}_0\rho_1)^2=id$. 
    Moreover, 
    $$\tilde{\rho}_0 = \rho_0\rho_1\rho_0 = \rho_1\rho_2\ldots\rho_{n-1}\rho_n\rho_{n-1}\ldots\rho_2\rho_1 U^{\frac{s-1}{2}} \rho_1 U^{-\frac{s-1}{2}} \rho_1\rho_2\ldots\rho_{n-1}\rho_n\rho_{n-1}\ldots\rho_2\rho_1,$$ which is a relation that can be derived from the fact that $$u^{-1}\tilde{\rho}_0u = (\rho_1\rho_2\ldots\rho_{n-1}\rho_n\rho_{n-1}\ldots\rho_3\rho_2)^2\rho_1.$$
    Finally, rewrite $U=(\tilde{\rho}_0\rho_2\rho_3\ldots\rho_{n-1}\rho_n\rho_{n-1}\ldots\rho_2\rho_1)^2$.
    With this, we can remove both $\rho_0$ and $U$ from $S$ and all the relations with $\rho_0$ from $R$. This gives the Coxeter group $[\nfrac{3}{3},3^{n-3},4] = \langle\tilde{\rho}_0,\rho_1,\ldots,\rho_n\rangle$ factored by the extra relation $$(\tilde{\rho}_0\rho_2\ldots\rho_{n-1}\rho_n\rho_{n-1}\ldots\rho_2\rho_1)^{2s} = id.$$

    \textbf{Case $k=n$ and both $s$ and $n$ are odd:}
    Let $u = \rho_0\rho_1\ldots\rho_{n-1}\rho_n$.
    Using a Tietze transformation, we can add the generators $\tilde{\rho}_0 = \rho_0\rho_1\rho_0$ and $U = (\rho_0\rho_1\ldots\rho_{n-1}\rho_n)^2$ to $S$. Notice that $U^{\frac{ns-1}{2}} (= u^{-1})$ is a word written with even number of $\rho_0$.
    As both $n$ and $s$ are odd, we can rewrite the extra relation in $R$ as $U^{ns}=id$.
    
    We can rewrite that $\rho_0 = \rho_1\rho_2\ldots\rho_{n-1}\rho_n U^{\frac{ns-1}{2}}$, which implies that the relations $(\rho_0\rho_i)^2 = id$ can be rewritten as $(U^{-1}\rho_i)^2=id$, for $i\geq 2$. However, as we have that $\rho_{i-1}u\rho_i=u^{-1}$, for $i\geq 2$, the relations can be deduced $(U^{-1}\rho_i)^2=id$ from the remaining ones, and hence we will removed from $R$.
    
    For $i\geq 3$, conjugating the relations $(\rho_1\rho_i)^2=id$ and $(\rho_1\rho_2)^3$ by $\rho_0$ gives us relations $(\tilde{\rho}_0\rho_i)^2=id$ and $(\tilde{\rho}_0\rho_2)^3=id$. The relation $(\rho_0\rho_1)^4=id$ can be rewritten as $(\tilde{\rho}_0\rho_1)^2=id$. 
    Moreover, $$\tilde{\rho}_0 = \rho_0\rho_1\rho_0 = U^{-\frac{ns-1}{2}}(\rho_n\rho_{n-1}\ldots\rho_{2}\rho_1)  \rho_1  (\rho_1\rho_2\ldots\rho_{n-1}\rho_n) U^{\frac{s-1}{2}},$$ which is a relation that can be derived from the fact that $$u^{-1}\tilde{\rho}_0u = \rho_n\rho_{n-1}\ldots\rho_{2}\rho_1 \rho_2\ldots\rho_{n-1}\rho_n.$$
    Finally, rewrite $U=\tilde{\rho}_0\rho_2\rho_3\ldots\rho_{n-1}\rho_n\rho_1\rho_2\ldots\rho_{n-1}\rho_n$.
    With this, we can remove both $\rho_0$ and $U$ from $S$ and all the relations with $\rho_0$ from $R$. This gives the Coxeter group $[\nfrac{3}{3},3^{n-3},4] = \langle\tilde{\rho}_0,\rho_1,\ldots,\rho_n\rangle$ factored by the extra relation $$(\tilde{\rho}_0\rho_2\rho_3\ldots\rho_{n-1}\rho_n\rho_1\rho_2\ldots\rho_{n-1}\rho_n)^{ns} = id.$$

\end{proof}
We now handle the case where $\Gamma[0,1]$ is bipartite. Here the automorphism group changes, as its cardinality is divided by two.

\begin{prop}\label{grouppres_even}
    Let either $k$ or $s$ be even and $\tilde{\rho}_0 = \rho_0\rho_1\rho_0$. Then $\Gamma[0,1]$ is bipartite and the automorphism group of $H(\Gamma)(0,1)$ will be a quotient of the Coxeter group $[\nfrac{3}{3},3^{n-3},4] = \langle\tilde{\rho}_0,\rho_1,\ldots,\rho_n\rangle$ by the following extra relation
    $$ (\tilde{\rho}_0\rho_2\ldots\rho_{n-1}\rho_n\rho_{n-1}\ldots\rho_2\rho_1)^{s} = id $$
    for $k=1$, and 
    $$ (\tilde{\rho}_0\rho_2\ldots\rho_{n-1}\rho_n\rho_{n-1}\ldots\rho_{k+1}\rho_k\rho_1\rho_2\ldots\rho_{n-1}\rho_n\rho_{n-1}\ldots\rho_{k+1}\rho_k)^{ks/2}=id$$
    for $k\in\{2,n\}$.
\end{prop}
\begin{proof}
    Let $H(\Aut(\Gamma)) := \langle \rho_0\rho_1\rho_0, \rho_1, \rho_2, \ldots, \rho_n\rangle = \langle \tilde{\rho}_0,\rho_1,\ldots,\rho_n\rangle $ be the halving group of automorphism group of the hypertope $\Gamma$. By Theorem~\ref{thm:main2} we have that $H(\Aut(\Gamma))=\Aut(H(\Gamma))$. Moreover, if either $k$ or $s$ is even, all relators of the group presentation of $\Aut(\Gamma)=G$ contain an even number of $\rho_0$'s and, by Proposition~\ref{bipartite_cond}, $H(\Aut(\Gamma))$ is a proper subgroup of $G$.

    Let us consider the case where $k=1$, $s$ is even and $H = H(\Aut(\Gamma))$. To determine the group presentation of the subgroup $H$ we will use the Reidemeister-Schreier subgroup presentation \cite{magnus2004combinatorial,conder1992schreier} using the Schreier coset graph of $H$. The subgroup $H$ is a subgroup of index $2$ of $\Aut(\Gamma)$, with the following distinct right cosets $\{H, H\rho_0\}$.
    We will construct the Schreier coset graph of $H$ in $\Aut(\Gamma)$, labelling the generator $s_{i,H}$ if the generator $\rho_i$ acts on the coset $H$ and $s_{i,H\rho_0}$ if the generator $\rho_i$ acts on the coset $H\rho_0$.

    $$
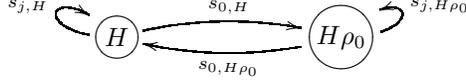
\xymatrix@-0.7pc{
    *++[o][F-]{H}\ar@{->}@/^/[rrr]^{s_{0,H}}\ar@{<-}@/_/[rrr]_{s_{0,H\rho_0}}\ar@(l,ul)[]^{s_{j,H}} &&& *++[o][F-]{H\rho_0}\ar@(r,ur)[]_{s_{j,H\rho_0}}
    }$$
    In the graph above, consider $j\in\{1,\ldots,n\}$, as the action of $\rho_j$ fixes the cosets $\{H,H\rho_0\}$.

    We can rewrite the relators of $G$ using the above generators of the graph as written below:
    \begin{equation*}
   \begin{split}
    s_{0,H}s_{0,H\rho_0}& = id,\\
    (s_{j,H})^2 = id, (s_{j,H\rho_0})^2& = id, \textnormal{ (for $j\in\{1,\ldots,n\}$)},\\
    (s_{0,H}s_{1,H\rho_0}s_{0,H\rho_0}s_{1,H})^2 = id, (s_{0,H\rho_0}s_{1,H}s_{0,H}s_{1,H\rho_0})^2& = id,\\
    (s_{0,H}s_{i,H\rho_0}s_{0,H\rho_0}s_{i,H}) = id, (s_{0,H\rho_0}s_{i,H}s_{0,H}s_{i,H\rho_0})& = id, \textnormal{ (for $i\in\{2,\ldots,n\}$)},\\
    (s_{1,H}s_{2,H})^3 = id, (s_{1,H\rho_0}s_{2,H\rho_0})^3& = id, \\ 
    (s_{1,H}s_{p,H})^2 = id, (s_{1,H\rho_0}s_{p,H\rho_0})^2& = id, \textnormal{ (for $p\in\{3,\ldots,n\}$)},\\
    (s_{q,H}s_{q+1,H})^3 = id, (s_{q,H\rho_0}s_{q+1,H\rho_0})^3& = id, \textnormal{ (for $q\in\{2,\ldots,n-1\}$)}\\ 
    (s_{q,H}s_{p,H})^2 = id, (s_{q,H\rho_0}s_{p,H\rho_0})^2& = id, \textnormal{ (for $p,q\in\{2,\ldots,n\}$},\\
    & \textnormal{and $q\notin\{p-1,p+1\}$})\\
    (s_{0,H}s_{1,H\rho_0}s_{2,H\rho_0}\ldots s_{n,H\rho_0}\ldots s_{2,H\rho_0}s_{1,H\rho_0}& \\
    s_{0,H\rho_0}s_{1,H}s_{2,H}\ldots s_{n,H}\ldots s_{2,H}s_{1,H})^{s/2}& =id,\\
    (s_{0,H\rho_0}s_{1,H}s_{2,H}\ldots s_{n,H}\ldots s_{2,H}s_{1,H}&\\
    s_{0,H}s_{1,H\rho_0}s_{2,H\rho_0}\ldots s_{n,H\rho_0}\ldots s_{2,H\rho_0}s_{1,H\rho_0})^{s/2}&=id.\\
    \end{split}
    \end{equation*}

    Consider the spanning tree in the coset graph by choosing the edge labelled $s_{0,H}$. Hence, by the Reidemeister-Schreier algorithm, $s_{0,H}=id$, which leads to $s_{0,H\rho_0}=id$ and $s_{i,H}=s_{i,H\rho_0}$, for $i\in\{2,\ldots,n\}$. The remaining generators form a set of Schreier generators of $H$. By relabelling the generators $s_{i,H}=s_{i,H\rho_0}=:\rho_i$, for $i\in\{2,\ldots,n\}$, $\rho_1 := s_{1,H}$ and $\tilde{\rho}_0 := s_{1,H\rho_0}$ and using Tietze transformations to remove conjugate relators, we get to the presentation given in the statement of this proposition.

    The same can be done for $k=2$ and $k=n$ (with either $s$ or $n$ even) just by changing the last two relators written above with the respective ones using the generators on the graph.
    
\end{proof}
As a direct consequence of these presentation, we obtain the following.

\begin{coro}\label{bipartitenn-1trunc}
    $H(\Gamma)(0,1)$ has a bipartite $\{n,n-1\}$-truncation.
\end{coro}
\begin{proof}
    This is a direct consequence of the group presentations of $H(\Gamma)(0,1)$ given in Propositions~\ref{grouppres_odd} and ~\ref{grouppres_even}, and Proposition~\ref{bipartite_cond}.
\end{proof}
Finally, we compute a presentation for the group of $H(\Gamma)(0,1)(n,n-1)$.
\begin{prop}
    Let $n>4$, $\tilde{\rho}_0 = \rho_0\rho_1\rho_0$ and $\tilde{\rho}_n = \rho_n\rho_{n-1}\rho_n$. Then the automorphism group of $H(\Gamma)(0,1)(n,n-1)$ is the quotient of the Coxeter group $[\nfrac{3}{3},3^{n-4},\nfrac{3}{3}] = \langle \tilde{\rho}_0, \rho_1,\ldots, \rho_{n-1}, \tilde{\rho}_n\rangle$ by the following extra relation
    $$ (\tilde{\rho}_0\rho_2\rho_3\ldots \rho_{n-1}\tilde{\rho}_n\rho_{n-2}\rho_{n-3}\ldots\rho_2\rho_1)^{2s} = id $$
    for $k=1$ and $s$ odd;
    $$ (\tilde{\rho}_0\rho_2\ldots\rho_{n-1}\tilde{\rho}_n\rho_{n-2}\ldots\rho_2\rho_1)^{s} = id $$
    for $k=1$ and $s$ even, 
    $$ (\tilde{\rho}_0\rho_2\ldots\rho_{n-1}\tilde{\rho}_n\rho_{n-2}\ldots\rho_{3}\rho_2\rho_1\rho_2\ldots\rho_{n-1}\tilde{\rho}_n\rho_{n-2}\ldots\rho_{3}\rho_2)^{s}=id$$
    for $k = 2$,
    $$ (\tilde{\rho}_0\rho_2\ldots\rho_{n-2}\rho_{n-1}\rho_1\rho_2\ldots\rho_{n-2}\tilde{\rho}_n)^{ns}=id$$
    for $k=n$ and both $s$ and $n$ odd, and
    $$ (\tilde{\rho}_0\rho_2\ldots\rho_{n-2}\rho_{n-1}\rho_1\rho_2\ldots\rho_{n-2}\tilde{\rho}_n)^{ns/2}=id$$
    for $k = n$ with either $s$ or $n$ even.
\end{prop}
\begin{proof}
    By Corollary~\ref{bipartitenn-1trunc}, we have that $H(\Gamma)(0,1)$ has a bipartite $\{n,n-1\}$-truncation. 
    Moreover, by Proposition~\ref{b1b2forhalvedcubic} we have that the $(n,n-1)$-leaf satisfies $(B_1)$ and $(B_2)$ in $H(\Gamma)(0,1)$.
    Hence, by Proposition~\ref{bipartite_cond} we have that the automorphism group of $H(\Gamma)(0,1)(n,n-1)$ will be a proper subgroup of the automorphism group of $H(\Gamma)(0,1)$. Hence, we can use the Reidemeister-Schreier subgroup presentation as shown above to arrive to the results in the statement of the proposition.
\end{proof}

For $n\in\{3,4\}$, consider the cases below.

\begin{prop}
    Let $n=3$, $\tilde{\rho}_0 = \rho_0\rho_1\rho_0$ and $\tilde{\rho}_n = \rho_n\rho_{n-1}\rho_n$. Then the automorphism group of $H(\Gamma)(0,1)(n,n-1)$ is the quotient of the Coxeter group $[(3,3,3,3)] = \langle \tilde{\rho}_0, \rho_1,\rho_2, \tilde{\rho}_3\rangle$ by the following extra relation
    $$ (\tilde{\rho}_0\rho_2\tilde{\rho}_3\rho_1)^{2s} = id $$
    for $k=1$ and $s$ odd;
    $$ (\tilde{\rho}_0\rho_2\tilde{\rho}_3\rho_1)^{s} = id $$
    for $k=1$ and $s$ even, 
    $$ (\tilde{\rho}_0\rho_2\rho_1\tilde{\rho}_3\rho_1\rho_2)^{s}=id$$
    for $k = 2$,
    $$ (\tilde{\rho}_0\rho_2\rho_1\tilde{\rho}_3)^{3s}=id$$
    for $k=3$ and $s$ odd, and
    $$ (\tilde{\rho}_0\rho_2\rho_1\tilde{\rho}_3)^{3s/2}=id$$
    for $k = 3$ and $s$ even.
\end{prop}

\begin{prop}
    Let $n=4$, $\tilde{\rho}_0 = \rho_0\rho_1\rho_0$ and $\tilde{\rho}_n = \rho_n\rho_{n-1}\rho_n$. Then the automorphism group of $H(\Gamma)(0,1)(n,n-1)$ is the quotient of the Coxeter group $[3^{1,1,1,1}] = \langle \tilde{\rho}_0, \rho_1,\rho_2, \rho_3, \tilde{\rho}_4\rangle$ by the following extra relation
    $$ (\tilde{\rho}_0\rho_2\rho_3\tilde{\rho}_4\rho_2\rho_1)^{2s} = id $$
    for $k=1$ and $s$ odd;
    $$ (\tilde{\rho}_0\rho_2\rho_3\tilde{\rho}_4\rho_2\rho_1)^{s} = id $$
    for $k=1$ and $s$ even, 
    $$ (\tilde{\rho}_0\rho_2\rho_3\tilde{\rho}_4\rho_2\rho_1\rho_2\rho_3\tilde{\rho}_4\rho_2)^{s}=id$$
    for $k = 2$, and
    $$ (\tilde{\rho}_0\rho_2\rho_3\rho_1\rho_2\tilde{\rho}_4)^{4s/2}=id$$
    for $k = 4$.
\end{prop}

\subsection{Diagrams of the halving geometries}\label{subsec:diagrams}\hfill
\\

From the previous results, we have that, starting with any polytope $\Gamma = \{4,3^{n-2},4\}_\textbf{s}$ where $\textbf{s} \neq (2,0^{n-1})$, we can apply the halving geometry once on each side of the diagram of $\Gamma$ for $n\geq 3$.
For $n\geq 5$ and $\textbf{s} \notin \{(2,0^{n-1}),(2,2,0^{n-2})\}$, we apply the halving geometry twice on each side of the diagram of $\Gamma$. In the following figures, we show the different diagrams obtained, up to duality, by successively applying the halving geometry to the leafs of the cubic toroids or their resulting geometries, for $n=3$ (Figure~\ref{fig:434}), $n=4$ (Figure~\ref{fig:4334}), $n=5$ (Figure~\ref{fig:43334}) and $n\geq 6$ (Figure~\ref{fig:43...34}).
It is not hard to check that all the halving operations can indeed be performed, even in the cases in which we did not explicitly prove that the leafs are still non-degenerate, even after multiple applications of the halving operation.

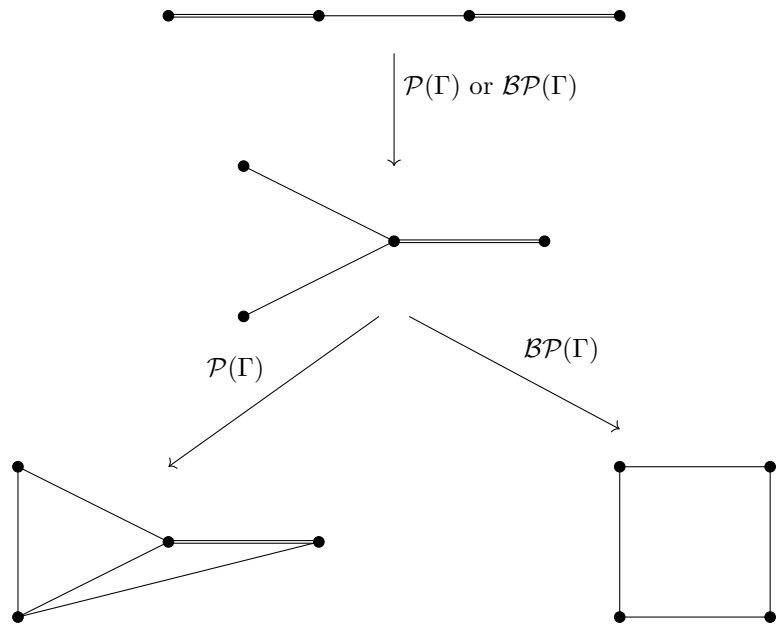
\begin{figure}[h!]
    \centering
    \begin{tikzpicture}
        \draw[double] (-2,0) -- (0,0);
        \draw (0,0) -- (2,0);
        \draw[double] (2,0) -- (4,0);
        \filldraw[black](0,0) circle(2pt);
        \filldraw[black](2,0) circle(2pt);
        \filldraw[black](-2,0) circle(2pt);
        \filldraw[black](4,0) circle(2pt);
        
        \draw[->] (1,-0.5)--(1,-2) node[midway,above right] {$\PP(\Gamma)$ or $\BP(\Gamma)$};
        
        \draw (-1,-2)--(1,-3) -- (-1,-4);
        \draw[double] (1,-3) -- (3,-3);
        \filldraw[black](1,-3) circle(2pt);
        \filldraw[black](3,-3) circle(2pt);
        \filldraw[black](-1,-2) circle(2pt);
        \filldraw[black](-1,-4) circle(2pt);

        \draw[->] (1.2,-4)--(4,-5.5) node[midway,above right] {$\BP(\Gamma)$};

        \draw[->] (0.8,-4)--(-2,-6) node[midway,above left] {$\PP(\Gamma)$};
        
        \begin{scope}[shift = {(-3,-4)}]
            \draw (-1,-2)--(1,-3) -- (-1,-4);
            \draw (-1,-2)--(-1,-4)--(3,-3);
            \draw[double] (1,-3) -- (3,-3);
            \filldraw[black](1,-3) circle(2pt);
            \filldraw[black](3,-3) circle(2pt);
            \filldraw[black](-1,-2) circle(2pt);
            \filldraw[black](-1,-4) circle(2pt);
        \end{scope}
        \begin{scope}[shift = {(+4,-6)}]
            \draw (0,-0)--(2,0) -- (2,-2) --(0,-2)--(0,0);
            \filldraw[black](0,0) circle(2pt);
            \filldraw[black](2,0) circle(2pt);
            \filldraw[black](2,-2) circle(2pt);
            \filldraw[black](0,-2) circle(2pt);
        \end{scope}
    \end{tikzpicture}
    \caption{The diagrams of the regular hypertopes obtained from the polyhedra of type $\{4,3,4\}_\textbf{s}$ applying the halving geometry to its leafs.}
    \label{fig:434}
\end{figure}

\begin{figure}
    \centering
    \begin{tikzpicture}
        \draw[double] (-4,0) -- (-2,0);
        \draw (-2,0)--(0,0) -- (2,0);
        \draw[double] (2,0) -- (4,0);
        \filldraw[black](0,0) circle(2pt);
        \filldraw[black](2,0) circle(2pt);
        \filldraw[black](-2,0) circle(2pt);
        \filldraw[black](-4,0) circle(2pt);
        \filldraw[black](4,0) circle(2pt);

        \draw[->] (0,-0.5)--(0,-2) node[midway,above right] {$\PP(\Gamma)$ or $\BP(\Gamma)$};
        
        \begin{scope}[shift = {(-2,0)}]
        \draw (-1,-2)--(1,-3) -- (-1,-4);
        \draw (1,-3)--(3,-3);
        \draw[double](3,-3)--(5,-3);
        \filldraw[black](1,-3) circle(2pt);
        \filldraw[black](3,-3) circle(2pt);
        \filldraw[black](-1,-2) circle(2pt);
        \filldraw[black](-1,-4) circle(2pt);
        \filldraw[black](5,-3) circle(2pt);
        \end{scope}

        \draw[->] (0.2,-4)--(3.5,-5.5) node[midway,above right] {$\BP(\Gamma)$};

        \draw[->] (-0.2,-4)--(-2,-6) node[midway,above left] {$\PP(\Gamma)$};
        
        \begin{scope}[shift = {(-5,-4.5)}]
            \draw (1,-1)--(1,-3);
            \draw (1,-1)--(-1,-3);
            \draw (1,-1)--(3,-3);
            \draw(-1,-3)-- (1,-3) -- (3,-3);
            \draw[double](3,-3)--(5,-3);
            \filldraw[black](1,-3) circle(2pt);
            \filldraw[black](3,-3) circle(2pt);
            \filldraw[black](-1,-3) circle(2pt);
            \filldraw[black](5,-3) circle(2pt);
            \filldraw[black](1,-1) circle(2pt);
        \end{scope}
        \begin{scope}[shift = {(0,-10)}]
            \draw (0,-0)--(2,0)--(0,-2);
            \draw (0,0) --(2,-2) --(0,-2)--(0,0)--(-2,-2);
            \draw (0,-2)--(-2,-2);
            \filldraw[black](0,0) circle(2pt);
            \filldraw[black](2,0) circle(2pt);
            \filldraw[black](2,-2) circle(2pt);
            \filldraw[black](0,-2) circle(2pt);
            \filldraw[black](-2,-2) circle(2pt);
        \end{scope}

        \draw[->] (4,-8)--(2.5,-9.5) node[midway,above left] {$\PP(\Gamma)$};

        \draw[->] (-2,-8)--(-0.5,-9.5) node[midway,above right] {$\BP(\Gamma)$};

        \begin{scope}[shift = {(+3,-4)}]
            \draw (-1,-2)--(1,-3)--(-1,-4);
            \draw(3,-2)--(1,-3)--(3,-4);
            \filldraw[black](1,-3) circle(2pt);
            \filldraw[black](3,-2) circle(2pt);
            \filldraw[black](-1,-2) circle(2pt);
            \filldraw[black](-1,-4) circle(2pt);
            \filldraw[black](3,-4) circle(2pt);
        \end{scope}
    \end{tikzpicture}
    \caption{The diagrams of the regular hypertopes obtained from the polytope of type $\{4,3,3,4\}_\textbf{s}$.}
    \label{fig:4334}
\end{figure}
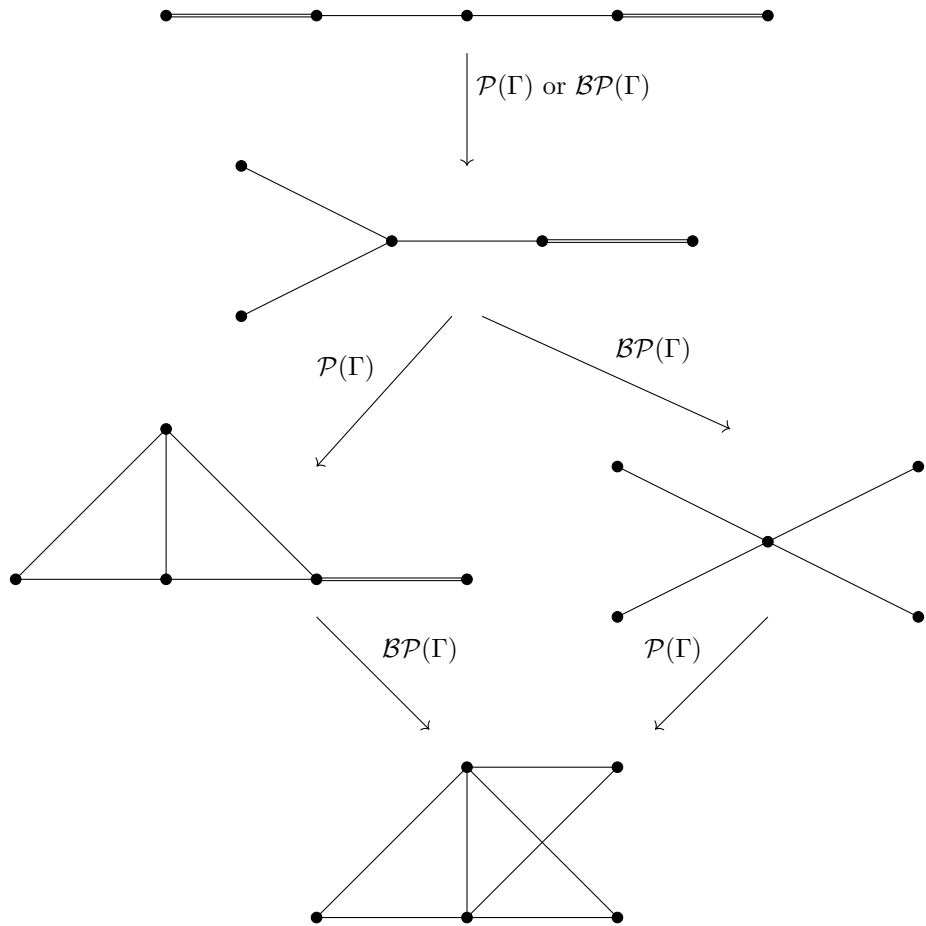

\begin{figure}
    \centering
    \begin{tikzpicture}
        \draw[double] (-4,0)--(-2,0);
        \draw (-2,0)--(0,0)--(2,0)--(4,0);
        \draw[double] (4,0)--(6,0);
        \filldraw[black](0,0) circle(2pt);
        \filldraw[black](2,0) circle(2pt);
        \filldraw[black](-2,0) circle(2pt);
        \filldraw[black](-4,0) circle(2pt);
        \filldraw[black](4,0) circle(2pt);
        \filldraw[black](6,0) circle(2pt);

        \draw[->] (1,-0.5)--(1,-2) node[midway,above right] {$\PP(\Gamma)$ or $\BP(\Gamma)$};
        
        \begin{scope}[shift = {(-1,-3)}]
            \draw (-2,1)--(0,0);
            \draw (-2,-1)--(0,0);
            \draw (0,0)--(2,0)--(4,0);
            \draw[double] (4,0)--(6,0);
            \filldraw[black](0,0) circle(2pt);
            \filldraw[black](2,0) circle(2pt);
            \filldraw[black](-2,1) circle(2pt);
            \filldraw[black](-2,-1) circle(2pt);
            \filldraw[black](4,0) circle(2pt);
            \filldraw[black](6,0) circle(2pt);
        \end{scope}

        \draw[->] (1,-4)--(1,-5) node[midway,above right]{$\BP(\Gamma)$};
        \draw[->,out=215,in=135]  (-4,-3) to node[midway,above right]{$\PP(\Gamma)$} (-4,-8) ;
        
        \begin{scope}[shift = {(0,-6)}]
            \draw (-2,1)--(0,0);
            \draw (-2,-1)--(0,0);
            \draw (0,0)--(2,0);
            \draw(4,1)-- (2,0) -- (4,-1);
            \filldraw[black](0,0) circle(2pt);
            \filldraw[black](2,0) circle(2pt);
            \filldraw[black](-2,1) circle(2pt);
            \filldraw[black](-2,-1) circle(2pt);
            \filldraw[black](4,1) circle(2pt);
            \filldraw[black](4,-1) circle(2pt);
            
        \end{scope}

    \draw[->,out=-45,in=45]  (5,-6) to node[midway,above right]{$\PP(\Gamma)$} (5,-13) ;
    
        \begin{scope}[shift = {(-4,-7)}]
            \draw (3,-1)--(3,-3)--(1,-3);
            \draw (5,-3)--(3,-1)--(1,-3);
            \draw (3,-1)--(3,-3);
            \draw (3,-3)--(5,-3)--(7,-3);
            \draw[double] (7,-3)--(9,-3);
            \filldraw[black](1,-3) circle(2pt);
            \filldraw[black](3,-3) circle(2pt);
            \filldraw[black](5,-3) circle(2pt);
            \filldraw[black](3,-1) circle(2pt);
            \filldraw[black](7,-3) circle(2pt);
            \filldraw[black](9,-3) circle(2pt);
        \end{scope}
        
           \draw[->] (1,-11)--(1,-12) node[midway,above right] {$\BP(\Gamma)$};
        
        \begin{scope}[shift = {(-3,-11)}]
            \draw (2,-1)--(2,-3);
            \draw (2,-1)--(0,-3);
            \draw (2,-1)--(4,-3);
            \draw (0,-3)-- (2,-3) -- (4,-3);
            \draw (6,-4)--(4,-3)--(6,-2);
            \filldraw[black](2,-3) circle(2pt);
            \filldraw[black](4,-3) circle(2pt);
            \filldraw[black](0,-3) circle(2pt);
            \filldraw[black](2,-1) circle(2pt);
            \filldraw[black](6,-4) circle(2pt);
            \filldraw[black](6,-2) circle(2pt);
            
        \end{scope}

          \draw[->] (1,-15)--(1,-16) node[midway,above right] {$\PP(\Gamma)$};
          
        \begin{scope}[shift = {(-3,-16)}]
            \draw (3,-1)--(3,-3);
            \draw (3,-1)--(3,-3);
            \draw (5,-1)--(3,-1)--(1,-3);
            \draw (3,-1)--(5,-3);
            \draw (1,-3)-- (3,-3) -- (5,-3);
            \draw (3,-3)--(5,-3)--(7,-3);
            \draw (5,-1)--(7,-3);
            \draw (5,-1)--(3,-3);
            \draw (5,-1)--(5,-3);
            \filldraw[black](1,-3) circle(2pt);
        \filldraw[black](3,-3) circle(2pt);
            \filldraw[black](5,-3) circle(2pt);
            \filldraw[black](3,-1) circle(2pt);
            \filldraw[black](7,-3) circle(2pt);
            \filldraw[black](5,-1) circle(2pt);
        \end{scope}

    \end{tikzpicture}
    \caption{The diagrams of the regular hypertopes obtained from the polytope of type $\{4,3,3,3,4\}_\textbf{s}$.}
    \label{fig:43334}
\end{figure}
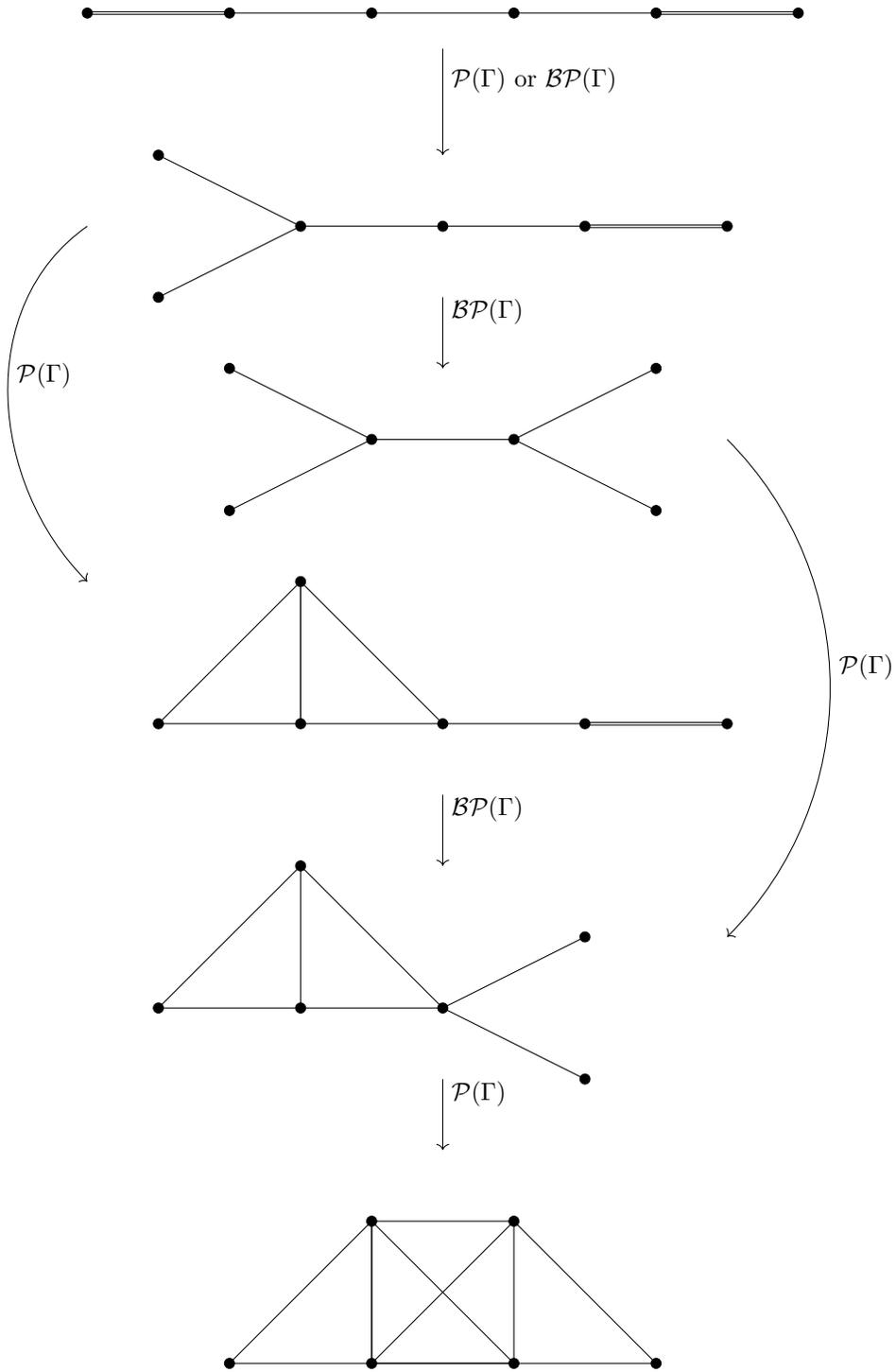

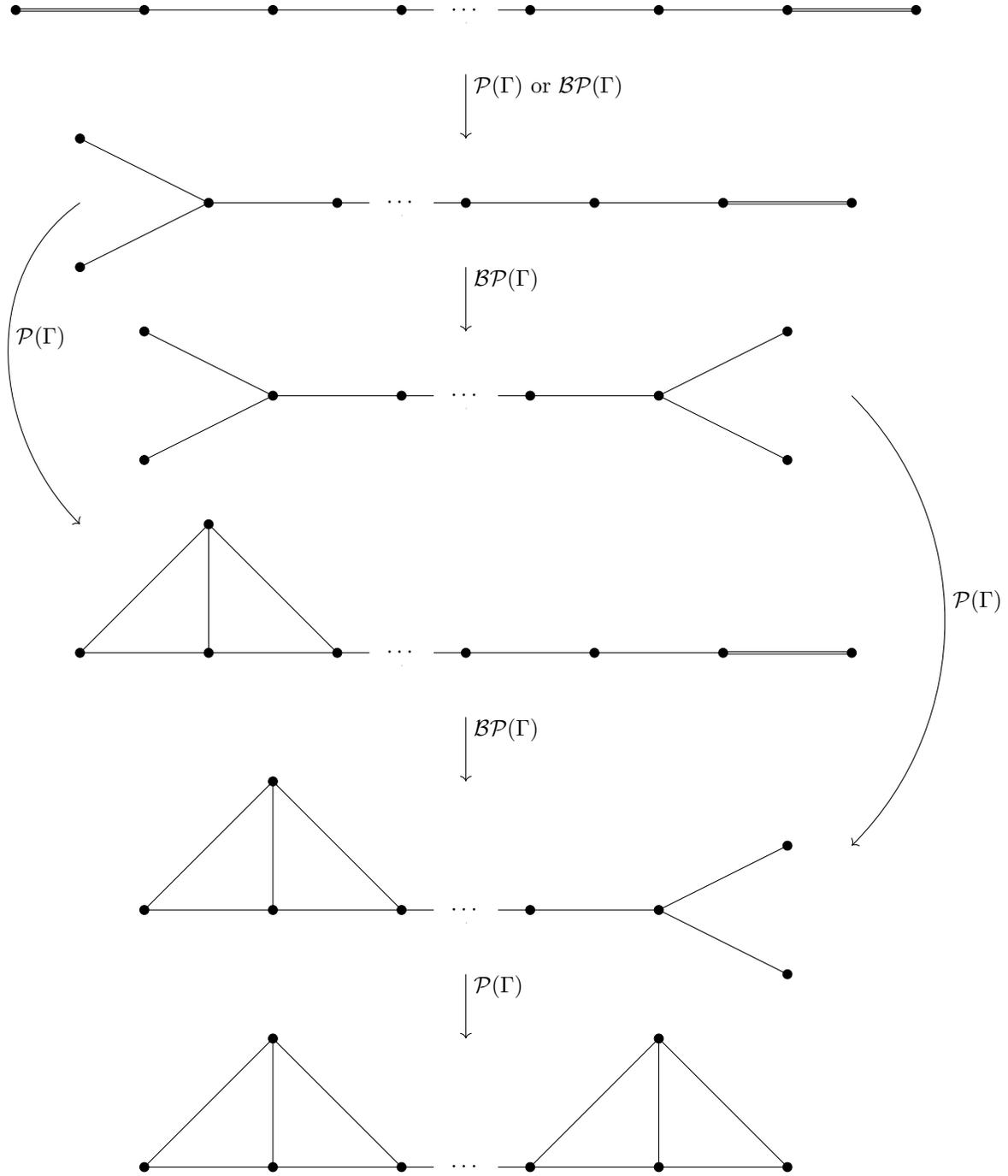
\begin{figure}
    \centering
    \begin{tikzpicture}
        \draw[double] (-6,0) -- (-4,0);
        \draw (-4,0)--(-2,0)--(0,0);
        \draw (0,0)--(0.5,0);
        \filldraw[black] (1,-0.2) circle (0pt)  node[anchor=south]{$\cdots$};
        \draw (1.5,0)--(2,0);
        \draw (2,0)--(4,0)--(6,0);
        \draw[double] (6,0) -- (8,0);
        \filldraw[black](0,0) circle(2pt);
        \filldraw[black](2,0) circle(2pt);
        \filldraw[black](-2,0) circle(2pt);
        \filldraw[black](-4,0) circle(2pt);
        \filldraw[black](4,0) circle(2pt);
        \filldraw[black](6,0) circle(2pt);
        \filldraw[black](-6,0) circle(2pt);
        \filldraw[black](8,0) circle(2pt);

        \draw[->] (1,-1)--(1,-2) node[midway,above right] {$\PP(\Gamma)$ or $\BP(\Gamma)$};
        
        \begin{scope}[shift = {(-1,-3)}]
            \draw (-4,1)--(-2,0);
            \draw (-4,-1)--(-2,0)--(0,0);
            \draw (0,0)--(0.5,0);
            \filldraw[black] (1,-0.2) circle (0pt)  node[anchor=south]{$\cdots$};
            \draw (1.5,0)--(2,0);
            \draw (2,0)--(4,0)--(6,0);
            \draw[double] (6,0) -- (8,0);
            \filldraw[black](0,0) circle(2pt);
            \filldraw[black](2,0) circle(2pt);
            \filldraw[black](-2,0) circle(2pt);
            \filldraw[black](-4,1) circle(2pt);
            \filldraw[black](-4,-1) circle(2pt);
            \filldraw[black](4,0) circle(2pt);
            \filldraw[black](6,0) circle(2pt);
            \filldraw[black](8,0) circle(2pt);
        \end{scope}

        \draw[->] (1,-4)--(1,-5) node[midway,above right]{$\BP(\Gamma)$};
        \draw[->,out=215,in=135]  (-5,-3) to node[midway,above right]{$\PP(\Gamma)$} (-5,-8) ;
        
        \begin{scope}[shift = {(0,-6)}]
            \draw (-4,1)--(-2,0);
            \draw (-4,-1)--(-2,0)--(0,0);
            \draw (0,0)--(0.5,0);
            \filldraw[black] (1,-0.2) circle (0pt)  node[anchor=south]{$\cdots$};
            \draw (1.5,0)--(2,0);
            \draw (2,0)--(4,0);
            \draw(6,1)-- (4,0) -- (6,-1);
            \filldraw[black](0,0) circle(2pt);
            \filldraw[black](2,0) circle(2pt);
            \filldraw[black](-2,0) circle(2pt);
            \filldraw[black](-4,1) circle(2pt);
            \filldraw[black](-4,-1) circle(2pt);
            \filldraw[black](4,0) circle(2pt);
            \filldraw[black](6,1) circle(2pt);
            \filldraw[black](6,-1) circle(2pt);
            
        \end{scope}

    \draw[->,out=-45,in=45]  (7,-6) to node[midway,above right]{$\PP(\Gamma)$} (7,-13) ;
    
        \begin{scope}[shift = {(-4,-7)}]
            \draw (1,-1)--(1,-3);
            \draw (1,-1)--(-1,-3);
            \draw (1,-1)--(3,-3);
            \draw(-1,-3)-- (1,-3) -- (3,-3);
            \draw (3,-3)--(3.5,-3);
            \filldraw[black] (4,-3.2) circle (0pt)  node[anchor=south]{$\cdots$};
            \draw (4.5,-3)--(5,-3);
            \draw (5,-3)--(7,-3)--(9,-3);
            \draw[double](9,-3)--(11,-3);
            \filldraw[black](1,-3) circle(2pt);
            \filldraw[black](3,-3) circle(2pt);
            \filldraw[black](-1,-3) circle(2pt);
            \filldraw[black](5,-3) circle(2pt);
            \filldraw[black](1,-1) circle(2pt);
            \filldraw[black](7,-3) circle(2pt);
            \filldraw[black](9,-3) circle(2pt);
            \filldraw[black](11,-3) circle(2pt);
        \end{scope}
        
       \draw[->] (1,-11)--(1,-12) node[midway,above right] {$\BP(\Gamma)$};
        
        \begin{scope}[shift = {(-3,-11)}]
            \draw (1,-1)--(1,-3);
            \draw (1,-1)--(-1,-3);
            \draw (1,-1)--(3,-3);
            \draw(-1,-3)-- (1,-3) -- (3,-3);
            \draw (3,-3)--(3.5,-3);
            \filldraw[black] (4,-3.2) circle (0pt)  node[anchor=south]{$\cdots$};
            \draw (4.5,-3)--(5,-3);
            \draw (5,-3)--(7,-3);
            \draw (9,-4)--(7,-3)--(9,-2);
            \filldraw[black](1,-3) circle(2pt);
            \filldraw[black](3,-3) circle(2pt);
            \filldraw[black](-1,-3) circle(2pt);
            \filldraw[black](5,-3) circle(2pt);
            \filldraw[black](1,-1) circle(2pt);
            \filldraw[black](7,-3) circle(2pt);
            \filldraw[black](9,-4) circle(2pt);
            \filldraw[black](9,-2) circle(2pt);
        \end{scope}

          \draw[->] (1,-15)--(1,-16) node[midway,above right] {$\PP(\Gamma)$};
          
        \begin{scope}[shift = {(-3,-15)}]
            \draw (1,-1)--(1,-3);
            \draw (1,-1)--(-1,-3);
            \draw (1,-1)--(3,-3);
            \draw(-1,-3)-- (1,-3) -- (3,-3);
            \draw (3,-3)--(3.5,-3);
            \filldraw[black] (4,-3.2) circle (0pt)  node[anchor=south]{$\cdots$};
            \draw (4.5,-3)--(5,-3);
            \draw (5,-3)--(7,-3)--(9,-3);
            \draw (7,-1)--(9,-3);
            \draw (7,-1)--(5,-3);
            \draw (7,-1)--(7,-3);
            \filldraw[black](1,-3) circle(2pt);
            \filldraw[black](3,-3) circle(2pt);
            \filldraw[black](-1,-3) circle(2pt);
            \filldraw[black](5,-3) circle(2pt);
            \filldraw[black](1,-1) circle(2pt);
            \filldraw[black](7,-3) circle(2pt);
            \filldraw[black](9,-3) circle(2pt);
            \filldraw[black](7,-1) circle(2pt);
        \end{scope}

    \end{tikzpicture}
    \caption{The diagrams of the regular hypertopes obtained from the polytope of type $\{4,3^{n-2},4\}_\textbf{s}$ for $n$ at least $6$.}
    \label{fig:43...34}
\end{figure}

\clearpage

\section{Acknowledgments}
Both authors are grateful to Dimitri Leemans for his insights and for making them aware of the content of \cite{Leemans2000} at the beginning of this project.

The author Claudio Alexandre Piedade was partially supported by CMUP, member of LASI, which is financed by national funds through FCT -- Funda\c c\~ao para a Ci\^encia e a Tecnologia, I.P., under the projects with reference UIDB/00144/2020 and UIDP/00144/2020.
This article was partially written while the first author was visiting the second author at the Max Planck Institute for Mathematics in the Sciences, in Leipzig.

\bibliographystyle{ieeetr} 
\bibliography{refs}

\end{document}